\documentclass{article}
\usepackage{amsmath, amsthm, amsfonts, amssymb}
\usepackage{mathpartir}
\usepackage{tikz}
\usepackage{tikz-cd}
\usepackage{rotating}

\newdimen\proofrulebreadth \proofrulebreadth=.05em
\newdimen\proofdotseparation \proofdotseparation=1.25ex
\newdimen\proofrulebaseline \proofrulebaseline=2ex
\newcount\proofdotnumber \proofdotnumber=3
\let\then\relax
\def\hfi{\hskip0pt plus.0001fil}
\mathchardef\squigto="3A3B
\newif\ifinsideprooftree\insideprooftreefalse
\newif\ifonleftofproofrule\onleftofproofrulefalse
\newif\ifproofdots\proofdotsfalse
\newif\ifdoubleproof\doubleprooffalse
\let\wereinproofbit\relax
\newdimen\shortenproofleft
\newdimen\shortenproofright
\newdimen\proofbelowshift
\newbox\proofabove
\newbox\proofbelow
\newbox\proofrulename
\def\shiftproofbelow{\let\next\relax\afterassignment\setshiftproofbelow\dimen0 }
\def\shiftproofbelowneg{\def\next{\multiply\dimen0 by-1 }%
\afterassignment\setshiftproofbelow\dimen0 }
\def\setshiftproofbelow{\next\proofbelowshift=\dimen0 }
\def\setproofrulebreadth{\proofrulebreadth}

\def\prooftree{% NESTED ZERO (\ifonleftofproofrule)
\ifnum  \lastpenalty=1
\then   \unpenalty
\else   \onleftofproofrulefalse
\fi
\ifonleftofproofrule
\else   \ifinsideprooftree
        \then   \hskip.5em plus1fil
        \fi
\fi
\bgroup% NESTED ONE (\proofbelow, \proofrulename, \proofabove,
\setbox\proofbelow=\hbox{}\setbox\proofrulename=\hbox{}%
\let\justifies\proofover\let\leadsto\proofoverdots\let\Justifies\proofoverdbl
\let\using\proofusing\let\[\prooftree
\ifinsideprooftree\let\]\endprooftree\fi
\proofdotsfalse\doubleprooffalse
\let\thickness\setproofrulebreadth
\let\shiftright\shiftproofbelow \let\shift\shiftproofbelow
\let\shiftleft\shiftproofbelowneg
\let\ifwasinsideprooftree\ifinsideprooftree
\insideprooftreetrue
\setbox\proofabove=\hbox\bgroup$\displaystyle % NESTED TWO
\let\wereinproofbit\prooftree
\shortenproofleft=0pt \shortenproofright=0pt \proofbelowshift=0pt
\onleftofproofruletrue\penalty1
}

\def\eproofbit{% NESTED TWO
\ifx    \wereinproofbit\prooftree
\then   \ifcase \lastpenalty
        \then   \shortenproofright=0pt  % 0: some other object, no indentation
        \or     \unpenalty\hfil         % 1: empty hypotheses, just glue
        \or     \unpenalty\unskip       % 2: just had a tree, remove glue
        \else   \shortenproofright=0pt  % eh?
        \fi
\fi
\global\dimen0=\shortenproofleft
\global\dimen1=\shortenproofright
\global\dimen2=\proofrulebreadth
\global\dimen3=\proofbelowshift
\global\dimen4=\proofdotseparation
\global\count255=\proofdotnumber
$\egroup  % NESTED ONE
\shortenproofleft=\dimen0
\shortenproofright=\dimen1
\proofrulebreadth=\dimen2
\proofbelowshift=\dimen3
\proofdotseparation=\dimen4
\proofdotnumber=\count255
}

\def\proofover{% NESTED TWO
\eproofbit % NESTED ONE
\setbox\proofbelow=\hbox\bgroup % NESTED TWO
\let\wereinproofbit\proofover
$\displaystyle
}%
\def\proofoverdbl{% NESTED TWO
\eproofbit % NESTED ONE
\doubleprooftrue
\setbox\proofbelow=\hbox\bgroup % NESTED TWO
\let\wereinproofbit\proofoverdbl
$\displaystyle
}%
\def\proofoverdots{% NESTED TWO
\eproofbit % NESTED ONE
\proofdotstrue
\setbox\proofbelow=\hbox\bgroup % NESTED TWO
\let\wereinproofbit\proofoverdots
$\displaystyle
}%
\def\proofusing{% NESTED TWO
\eproofbit % NESTED ONE
\setbox\proofrulename=\hbox\bgroup % NESTED TWO
\let\wereinproofbit\proofusing
\kern0.3em$
}

\def\endprooftree{% NESTED TWO
\eproofbit % NESTED ONE
  \dimen5 =0pt% spread of hypotheses
\dimen0=\wd\proofabove \advance\dimen0-\shortenproofleft
\advance\dimen0-\shortenproofright
\dimen1=.5\dimen0 \advance\dimen1-.5\wd\proofbelow
\dimen4=\dimen1
\advance\dimen1\proofbelowshift \advance\dimen4-\proofbelowshift
\ifdim  \dimen1<0pt
\then   \advance\shortenproofleft\dimen1
        \advance\dimen0-\dimen1
        \dimen1=0pt
        \ifdim  \shortenproofleft<0pt
        \then   \setbox\proofabove=\hbox{%
                        \kern-\shortenproofleft\unhbox\proofabove}%
                \shortenproofleft=0pt
        \fi
\fi
\ifdim  \dimen4<0pt
\then   \advance\shortenproofright\dimen4
        \advance\dimen0-\dimen4
        \dimen4=0pt
\fi
\ifdim  \shortenproofright<\wd\proofrulename
\then   \shortenproofright=\wd\proofrulename
\fi
\dimen2=\shortenproofleft \advance\dimen2 by\dimen1
\dimen3=\shortenproofright\advance\dimen3 by\dimen4
\ifproofdots
\then
        \dimen6=\shortenproofleft \advance\dimen6 .5\dimen0
        \setbox1=\vbox to\proofdotseparation{\vss\hbox{$\cdot$}\vss}%
        \setbox0=\hbox{%
                \advance\dimen6-.5\wd1
                \kern\dimen6
                $\vcenter to\proofdotnumber\proofdotseparation
                        {\leaders\box1\vfill}$%
                \unhbox\proofrulename}%
\else   \dimen6=\fontdimen22\the\textfont2 % height of maths axis
        \dimen7=\dimen6
        \advance\dimen6by.5\proofrulebreadth
        \advance\dimen7by-.5\proofrulebreadth
        \setbox0=\hbox{%
                \kern\shortenproofleft
                \ifdoubleproof
                \then   \hbox to\dimen0{%
                        $\mathsurround0pt\mathord=\mkern-6mu%
                        \cleaders\hbox{$\mkern-2mu=\mkern-2mu$}\hfill
                        \mkern-6mu\mathord=$}%
                \else   \vrule height\dimen6 depth-\dimen7 width\dimen0
                \fi
                \unhbox\proofrulename}%
        \ht0=\dimen6 \dp0=-\dimen7
\fi
\let\doll\relax
\ifwasinsideprooftree
\then   \let\VBOX\vbox
\else   \ifmmode\else$\let\doll=$\fi
        \let\VBOX\vcenter
\fi
\VBOX   {\baselineskip\proofrulebaseline \lineskip.2ex
        \expandafter\lineskiplimit\ifproofdots0ex\else-0.6ex\fi
        \hbox   spread\dimen5   {\hfi\unhbox\proofabove\hfi}%
        \hbox{\box0}%
        \hbox   {\kern\dimen2 \box\proofbelow}}\doll%
\global\dimen2=\dimen2
\global\dimen3=\dimen3
\egroup % NESTED ZERO
\ifonleftofproofrule
\then   \shortenproofleft=\dimen2
\fi
\shortenproofright=\dimen3
\onleftofproofrulefalse
\ifinsideprooftree
\then   \hskip.5em plus 1fil \penalty2
\fi
}

\newtheorem{thm}{Theorem}[section]
\newtheorem{prop}[thm]{Proposition}
\newtheorem{lem}[thm]{Lemma}
\newtheorem{cor}[thm]{Corollary}
\newtheorem{defn}[thm]{Definition}

\theoremstyle{remark}

\newcommand{\bear}{\begin{eqnarray*}}
\newcommand{\eear}{\end{eqnarray*}}
\newcommand{\bea}{\begin{eqnarray}}
\newcommand{\eea}{\end{eqnarray}}
\def\nn{\nonumber \\}
\def\Cat#1{\mathsf{#1}}
\def\Cal#1{\mathcal{#1}}
\def\calO{\mathcal{O}}
\def\modeledby{=\joinrel\mathrel|}

\def\RR{\mathbb{R}}
\def\TT{\mathbb{T}}

\def\la{\langle}
\def\ra{\rangle}
\newcommand{\T}{T_\ast}
\newcommand{\TS}{T^\ast}
\newcommand{\TCS}{T_\#}
\newcommand{\FS}{\Theta^\ast}
\newcommand{\F}{\Theta_\#}
\newcommand{\RRD}{\RR_\#}
\newcommand{\G}{\Sigma_\ast}

\def\calA{\mathcal{A}}
\def\calN{\mathcal{N}}   
\def\rmd{\mathrm{d}}
\newcommand{\Yoneda}[1]{\widehat{#1}}
\def\Lin{L}
\def\Cc{C}

\DeclareMathOperator{\bfD}{\mathbf{D}}
\def\bfDk{\bfD\mkern-2mu}

\def\Id{I\!d}
\newcommand{\dist}{d}

\renewcommand{\to}{\rightarrow}

\newcommand{\tto}[1]{\xrightarrow{#1}}
\newcommand{\oot}[1]{\xleftarrow{#1}}
\newcommand{\mono}{\rightarrowtail}
\newcommand{\epi}{\twoheadrightarrow}

\newcommand{\inclusion}{\hookrightarrow}

\mathcode`\<="4268 %left delimiter
\mathcode`\>="5269 %right delimiter
\mathchardef\gt="313E %relation >
\mathchardef\lt="313C %relation <

\newcommand{\AAA}{{\cal A}}

\newcommand{\EEE}{{\cal E}}

\newcommand{\OOO}{{\cal O}}

\newcommand{\SSS}{{\cal S}}
\newcommand{\TTT}{{\cal T}}
\newcommand{\UUU}{{\cal U}}

\newcommand{\Man}{\Cat{Man}}

\title{Smooth coalgebra:\\ testing vector analysis}
\author{Dusko Pavlovic\\
        University of Hawaii\\
        \texttt{dusko@hawaii.edu}
\and
Bertfried Fauser\\
University of Konstanz\\
        \texttt{Bertfried.Fauser@uni-konstanz.de}
}
\date{}
\begin{document}
%=======================================================================
\maketitle
\begin{abstract}
Processes are often viewed as coalgebras, with the structure maps
specifying the state transitions. In the simplest case, the state spaces
are discrete, and the structure map simply takes each state to the next
states. But the coalgebraic view is also quite effective for studying
processes over structured state spaces, e.g. measurable, or continuous.
In the present paper we consider coalgebras over manifolds. This means
that the captured processes evolve over state spaces that are not just
continuous, but also locally homeomorphic to normed vector spaces, and thus
carry a differential structure. Both dynamical systems and differential
forms arise as coalgebras over such state spaces, for two different
endofunctors over manifolds. A duality induced by these two endofunctors
provides a formal underpinning for the informal geometric intuitions
linking differential forms and dynamical systems in the various
practical applications, e.g. in physics. This joint functorial
reconstruction of tangent bundles and cotangent bundles uncovers the
universal properties and a high level view of these fundamental
structures, which are implemented rather intricately in their standard
form. The succinct coalgebraic presentation provides unexpected
insights even about the situations as familiar as Newton's laws.
\end{abstract}
%==will be commented out in final version===============================
%{\tiny \tableofcontents}
%=======================================================================
% Introduction
%-----------------------------------------------------------------------
% !TEX root = manifolds.tex

%=======================================================================
\section{Introduction}\label{sec:Introduction}
\subsubsection*{Idea: \emph{'Gedankenexperiments'}\/ with algebras and coalgebas}
The coalgebraic view of processes is based on capturing the state
transitions by the coalgebra structure maps. Since the coalgebra
homomorphisms then preserve and reflect the state transitions, they
capture the observable behaviors, and the elements of the final
coalgebras correspond to the equivalence classes of indistinguishable
behaviors. The main utility of the coalgebraic theory of processes has
been drawn from this correspondence \cite{RuttenJ:Universal}.

However, what is observable, and which behaviors can be distinguished,
is not always determined just by the transition systems, i.e. by the
coalgebraic specifications alone, like it is in the special case of
concurrent processes modulo strong bisimulation~\cite{AczelP:NWF,PavlovicD:CCPS1}.
Both in system design and in theoretical analyses, behaviors are usually
specified in various \emph{testing}\/ frameworks, along the lines of
E.F.~Moore's seminal \emph{'Gedankenexperiments'}\/ paper~\cite{Moore:Gedanken}.
E.g., in the theory of computation, a process can be specified as a
Turing machine, or as a finite state automaton, and such specifications
can be naturally presented as coalgebras. But to specify how this
process processes data, i.e. how does a state machine compute, and what
does it mean that it accepts a language, we must also say how it
interacts with the words representing data. And the words are
elements of algebras. The interactions between machines and words are
Moore's testing correlations. A categorical version of testing
correlations, where machines are presented as coalgebras, and tests
as the elements of algebras, was studied in~\cite{Pavlovic:Mislove:Worrell:2006a}.
Here we lift the same framework to dynamical systems presented as
coalgebras, and paths as the elements of algebras.

\subsubsection*{Background: Semantic connections of algebras and coalgebras}
It is essential to note that, in the testing approach to specifying
behaviors, both coalgebras and algebras are parameters of the
description, and neither side determines the other. E.g., the whole
linear time/branching time spectrum of essentially different computational
behaviors~\cite{vanGlabbeek-I,vanGlabbeek-II} can be described by
fixing the family of coalgebras, and varying the test algebras. On the
other hand, the Chomsky hierarchy of languages, as well as the hierarchy
of complexity classes, can be described by fixing the family of tests,
and varying the state machines, \emph{viz}\/ the corresponding
coalgebras. Both approaches were aligned
in~\cite{Pavlovic:Mislove:Worrell:2006a}\footnote{The latter approach
probably deserves a more detailed explanation, especially since its
details had to be moved to the Appendix, and the Appendix got removed
from the published version of the paper. The full submitted version
remained available online.}, generalizing the earlier application
in~\cite{PavlovicD:SSAS}. Such \emph{loose}\/ semantical connections
of algebras and coalgebras, as two independent dimensions of semantical
descriptions, distinguish the testing frameworks used in the present
paper, and previously applied
in~\cite{PavlovicD:SSAS,PavlovicD:ASE01,PavlovicD:AMAST02,Pavlovic:Mislove:Worrell:2006a,PavlovicD:AMAST08,PavlovicD:MPC10},
from the \emph{tight}\/ semantical connections of algebras and
coalgebras, as arising on the two sides of a duality, and used in
algebraic semantics of coalgebraic
logic~\cite{Kupke:Kurz:Pattinson:04,Kupke:Kurz:Pattinson:05,Kurz:Rosicky:12}. 
The idea of testing is echoed more closely in the testing approach
to the equivalence of concurrent processes~\cite{Nicola:Hennessy:84,DeNicola:2012},
and the two dimensional approach is implemented in terms of algebras and
coalgebras in the categorical approach to Structured Operational Semantics,
which was developed as an extension of Denotational Semantics of
monads~\cite{Plotkin:Turi,KlinB:07,KlinB:09}. However, while
algebras and coalgebras do not completely constrain each other, and varying the algebras allows capturing the linear time/branching time spectrum, they were specified within the same category, with their semantical connection hardwired through distributivity requirements precluding, e.g., capturing the language hierarchies. On the other hand, semantic connections of algebras and coalgebras were studied in a wide variety of frameworks, and by a wide variety of techniques \cite[to mention just a few]{KlinB:beyond,Jacobs-Sokolova-exemplaric,Kapulkin}, and it is possible that the language hierarchies, language acceptance, and computability and complexity concepts could have been captured in that framework if the community moved in that direction.
The language acceptance relation and the trace equivalences have in fact been captured in~\cite{Hasuo-Jacobs-Sokolova}, but by combining coalgebras and monads, which seem to indirectly capture the underlying semantical connection. E.g., capturing the linear time -- branching time spectrum in that framework seems quite a bit more involved than through testing. Last but not
least, the general paradigm of testing, explicated by E.F.~Moore, has been widely used before and independently of his work, not only in concurrency theory~\cite{DeNicola:2012}, but also, e.g., in functional analysis, as the foundation of theory of distributions~\cite{Schwarz:1950a}. In fact, the idea to apply testing in the context of manifolds, that we propose in the present paper, could have just as easily emerged from the theory of
distributions as from the theory of coalgebras. A clearly
coalgebraic view of differential geometry can be traced back to the
1970s work of Modungo and Stefani on the second tangent bundle and
its mixed versions~\cite{Modungo:Stefani:1978a}. The naturality of
the bundle structures was explicated in~\cite{Kolar:Michor:Slovak:1993a}.
A coalgebraic view of tangent bundles was proposed
in~\cite{Haghiverdi-Pappas}, leading to the idea that dynamical
systems could be viewed modulo the bisimulations relations. This however
may suffer from too strong assumptions about a comonad structure on the
tangent functor, see comments below.

A different categorical approach to tangent bundles, framed in the
context of Synthetic  Differential Geometry (SDG)~\cite{KockA:SDG},
goes back to Lawvere~\cite{LawvereFW:smooth} and Rosick\'y~\cite{Rosicky:1984a}.
The idea of SDG is to derive the important constructions of
differential geometry from abstract categorical axioms. The analytic
framework in which differential geometry is usually done is just
one of the models of SDG. In order to open an alley towards
differential geometry in nonstandard categorical model categories,
SDG attempts to extract and axiomatize what is essential for
differential geometry, and to abstract away the inessential
implementation details. This is, of course, a clean and exciting
approach, but it involves researchers' design decisions what is
essential in their theory, and what not. Many scientific
discoveries, however, happen to emerge from the parts of scientific
theories which have been previously thought to be inessential. This
keeps scientific practice from being superseded by axiomatic
theories. Nevertheless, axiomatic theories continue to play their
role as well, and categorical axiomatics of differential geometry
have in the meantime grown into a rich, active, and useful research
area~\cite{Blute:Cockett:Seely:differential,Blute:Cockett:Seely:cart-differential,Cockett:Seely:2011a,Cockett:Cruttwell:2012a}. 

Lastly, we also mention the work on entwining structures, which
has led our intuitions at various points, although its technical links
with this work are less clear. Entwining structures are widely used
in mathematical physics to study modules of quantum
deformations~\cite{Caenpeel:Militaru:Zhu:2002a}. A categorical
description can be found in the papers by Mesablishvili and
Wisbauer~\cite{mesablishvili:wisbauer:2010a,mesablishvili:wisbauer:2011a}.
Moerdijk's definition of a Hopf monad in~\cite{moerdijk:2002a} seems
to have been an important technical step in these analyses, with
possible repercussions on our own work that remain to be explored.

\subsubsection*{Outline of the paper}
In Section~\ref{sec:generaltesting} we provide a general overview of
testing correlations. In Section~\ref{sec:manifolds}, we summarize the
basic ideas about manifolds and their tangents. In
Section~\ref{sec:tangenttesting}, we apply testing correlations in
the context of manifolds to provide the semantic reconstructions of the
of the tangent bundle functors and of the contangent bundle functors,
spelled out in Theorem~\ref{thm-main}. The coalgebras for these endofunctors
are the usual \emph{cross sections} of the bundle projections, and
thus respectively correspond to vector fields (or dynamical systems)
and to differential forms. The testing correlations over manifolds thus
provides a categorical view of the practice of integration of differential
systems over differential forms. In Section~\ref{sec:Newton},
we spell out Newton's Second Law as an example of the
coalgebraic treatment. Interestingly, the structural duality of the
tangent and the cotangent bundles, displayed in the categorical
treatment, immediately points beyond Newton, and into relativity theory.
An overview of the standard definitions from vector analysis is provided in
Appendix~\ref{app:Banach}. 

%=======================================================================
% General testing framework
%-----------------------------------------------------------------------
% !TEX root = manifolds.tex
%=======================================================================
\section{General testing framework}\label{sec:generaltesting}
%=======================================================================
%=======================================================================
We begin by reviewing the testing framework from~\cite{Pavlovic:Mislove:Worrell:2006a}. 

\subsection{Idea} \label{subsec:testingRel}
Given a family of systems $\Sigma$, a family of tests $\Theta$, and a
type $\Omega$ of observations, we call a map
\begin{align}\label{correlation}
	\begin{tikzcd}[ampersand replacement=\&]
		\Sigma \times \Theta\arrow{r}{\TT} \& \Omega
	\end{tikzcd}
\end{align}
a \emph{testing correlation}, or just \emph{testing}. The observation
$\TT(S,t)$ is often written in the infix form $S\models t$. The
observations can be boolean, like `true'/`false', or `pass'/`fail';
but they can also be quantities obtained from a measurement, e.g. in
the interval $[0,1]$, or on the real line $\RR$. Each test is assumed
to yield a single observation. In the simplest case, we may use testing
to distinguish a given system $S\in\Sigma$ from a reference system
$R\in \Sigma$. The two systems are observably different if there is a
test $b\in \Theta$ such that the observation $R\models b$ is different
from the observation $S\models b$. Otherwise, if the two systems induce
the same observations for all tests, then they are
\emph{observationally indistinguishable}, and we write 
\begin{align*}
S \sim R &\Longleftrightarrow \forall t\in\Theta.(S\models t) = (R\models t)
\end{align*}
Developing this idea in \cite{Moore:Gedanken}, E.F.~Moore suggested that
minimal representations of automata can be built over the equivalence
classes of their observationally indistinguishable states. This idea
was elaborated categorically in~\cite{Pavlovic:Mislove:Worrell:2006a},
by identifying each equivalence class of systems that are observationally
indistinguishable from $S\in \Sigma$ with the map $S\models(-): \Theta \to \Omega$.{}
Such maps can be thought of as the \emph{observable behaviors}\/ of
systems. The family $L$ of observable behaviors of systems from
$\Sigma$ can thus be obtained as the image in $\Omega^\Theta$ of the
mapping that sends each system $S$ to the induced function
$S\models(-)$ from tests to observations:
\begin{align}\label{factor}
	\begin{tikzcd}[ampersand replacement=\&,column sep=small]
		\Sigma \arrow{rr}{\models}
		       \arrow[two heads]{rd} 
		\& \& \Omega^{\Theta} \\
		\& 
		L \arrow[tail]{ur} 
		\&
	\end{tikzcd}
\end{align} 
The main feature of this representation is that the elements of $L$,
with a suitable coalgebraic structure, can be used to build the
canonical minimal representatives of the behaviors of the systems in
$\Sigma$, in so far as they are observable under testing by the tests
from $\Theta$. Now we outline the construction of this coalgebraic
structure on $L$, as presented in~\cite{Pavlovic:Mislove:Worrell:2006a}.

\subsection{Semantic connections} 
Let $\Cal{S}$ be a category of `spaces' and $\Cal{T}$ be a category
of `types' or `theories'. A \emph{semantic connection}\/ between
$\Cal{S}$ and $\Cal{T}$ is simply a contravariant adjunction
$M^{op} \dashv P : \Cal{S}^{op}\rightarrow \Cal{T}$. The idea is that
a space $X\in \Cal{S}$ is mapped to the type $PX$ of \emph{predicates}
over it, and that a theory $A\in \Cal{T}$ is mapped to the space
$MA$ of its models. 

The examples abound. Stone duality motivates the logical view:
$\Cal{T}$ is the category of Boolean algebras, $\Cal{S}$ is the category
of Stone spaces, $M$ maps each Boolean algebra to its space of
ultrafilters, which can be viewed as Boolean homomorphisms to $2$,
whereas $P$ sends each Stone space to its Boolean algebra of open sets,
which can be viewed as continuous maps to $2$. If we restrict
$\Cal{S}$ from Stone spaces to the category of sets $\sf Set$,
viewed as topological spaces where every subset is open, then
$\Cal{T}$ restricts to the category of complete atomic Boolean
algebras, which is equivalent to the opposite category of sets, which
yields the self-adjunction of the powerset
$\wp^{op} \dashv \wp : {\sf Set}^{op} \to {\sf Set}$ as another
example of a semantic connection. For a different example, let
$\Cal{S}$ be the lattice $\wp \RR^n$ of sets of $n$-tuples of real
numbers, $\Cal{T}$ the lattice $\wp \RR[x_1,\ldots x_n]$ of sets of
real polynomials in $n$ variables, and let $M$ map each set of
polynomials to the $n$-tuples where they are all zero, whereas
$P$ should map each set of $n$-tuples to all polynomials that
vanish over it. When restricted to the fixed points of $MP$
and of $PM$, this semantic connection yields the duality of Zariski
closed subsets of $\RR^n$ and of radical ideals in $\RR[x_1,\ldots x_n]$.

When they have enough limits and colimits, the two sides of a semantic
connection share a \emph{dualizing object} $\Omega$, which lives in
both categories as the type $P1$ and the space $M1$ respectively. We
can then construct the arrows
\begin{align*}
	&\inferrule*{\coprod_{\vert X\vert}1\rightarrow X}{%
	     PX \rightarrow P(\coprod_{\vert X\vert}1)
	        \stackrel{\sim}{\rightarrow} \prod_{\vert X\vert}P1}
	&&&
	 \inferrule*{\coprod_{\vert A\vert}1\rightarrow A}{%
	     MA \rightarrow M(\coprod_{\vert A\vert}1)
	        \stackrel{\sim}{\rightarrow} \prod_{\vert A\vert}M1}        
\end{align*}
where $\vert C\vert = \Cat{C}(1,C)$, and define the functors
$\Omega^{X}:=\prod_{\vert X\vert}P1$ and
$\Omega^{A}:=\prod_{\vert A\vert}M1$, to get the natural embeddings
\begin{align*}
  &\begin{tikzcd}[ampersand replacement=\&,column sep=small]
		PX \arrow[tail]{r} \& \Omega^{X} 
	\end{tikzcd}
&&&\begin{tikzcd}[ampersand replacement=\&,column sep=small]
		MA \arrow[tail]{r} \& \Omega^{A} 
	\end{tikzcd}		
\end{align*}
 
%=======================================================================
\subsection{Testing coalgebras by process algebras}\label{subsec:proclogtestalg}\label{subsec:repTesting}
Process logics are modal logics for describing the behavior of
computational processes~\cite{PrattV:floyd,Milner:CC,Harel-Kozen}. With
the modalities capturing the actions of a process, the formulas of
process logics can be constructed as tests: a process \emph{satisfies}\/
a formula if and only if it passes the test that the formula represents.
Viewed like this, semantics of process logics generalize the
language acceptance definitions, used in the theory of computation to
specify how automata and Turing machines compute. In the setting of
semantic connections, semantics of process logics, including the various
language acceptance definitions, can be captured as follows.
\begin{itemize}
\item Present processes as coalgebras in $\Cal{S}$ for some endofunctor $G:\Cal{S}\to \Cal{S}$.
\item Present process logics as algebras in $\Cal{T}$ for some endofunctor $F:\Cal{T}\to \Cal{T}$.
\item Specify the $F$-algebra semantics of $G$-coalgebras as a distributivity
      law $FP\xrightarrow\lambda PG$: it lifts the final
      $G$-coalgebra $X\xrightarrow\xi GX$ into an $F$-algebra
      $FPX \xrightarrow \lambda PGX\xrightarrow{P\xi} PX$.
\item The $F$-algebra testing of $G$-coalgebras is realized as the algebra
      homomorphism from the initial $F$-algebra
      $FA\xrightarrow \alpha A$ to $FPX \xrightarrow \lambda PGX\xrightarrow{P\xi} PX$ in $\Cal T$.
\item Transposed along the semantic connection, this algebra homomorphism
      induces in $\Cal S$ a quotient $L$ of $X\xrightarrow \xi GX$, the
      finest (initial) among all of its quotients that embeds into
      $MA \xrightarrow{M\alpha} MFA$. 
\item As an $MFP$-coalgebra, $L$ provides the canonical representatives
      of $F$-observable behaviors of $G$-processes.
\end{itemize}
This procedure is our formalization `in the small' of the idea of
Eqn.~\eqref{factor}. The coalgebra $X\xrightarrow \xi GX$ represents
the family of systems $\Sigma$, the algebra $FA\xrightarrow \alpha A$
represents the family of tests $\Theta$, and the elements of
the dualizing object $\Omega$ are the available observations. This
formalization was elaborated in~\cite{Pavlovic:Mislove:Worrell:2006a}.
We quote the main theorem from that paper, omitting the considerations
regarding the initiality of $FA\xrightarrow \alpha A$ and the
finality of $X\xrightarrow \xi GX$.

\begin{thm}\label{thm:representation}
  For the 'model/predicate' adjunction $M^{op}\dashv P : \Cal{S}^{op}\rightarrow \Cal{T}$
  and endofunctors $G : \Cal{S}\rightarrow \Cal{S}$ and
  $F : \Cal{T}\rightarrow \Cal{T}$ and a distributive law
  $\lambda : FP \rightarrow PG$ the following holds:
  \begin{itemize}
\item[(a)]
  The predicate functor $P : \Cal{S}^{op} \rightarrow \Cal{T}$
  lifts to $\hat{P} : (\Cal{S}_{G})^{op} \rightarrow {}_{F}\Cal{T}$
  \begin{align}
	  \inferrule*[Right=lift]{
	  \begin{tikzcd}[ampersand replacement=\&]
		  X \arrow{r}{\xi} \& GX
	  \end{tikzcd} 
	  }{
	  \begin{tikzcd}[ampersand replacement=\&]
		  \hat{P}\xi : FPX \arrow{r}{\lambda} \& 
		  PGX \arrow{r}{P\xi}\&
		  PX
	  \end{tikzcd}
	  }
  \end{align}
\item[(b)]
  $\hat{P}$ has in general no adjoint, but there is a correspondence
  \begin{align*}
	  \inferrule*{
	  \begin{tikzcd}[ampersand replacement=\&]
		  \alpha \arrow{r} \& \hat{P}\xi
	  \end{tikzcd} 
	  }{
	  \begin{tikzcd}[ampersand replacement=\&]
        \Lambda\xi \arrow{r} \& M\alpha
	  \end{tikzcd}
	  }
  \end{align*}
  where $\Lambda : \Cal{S}_{G} \rightarrow \Cal{S}_{MFP}$ is the functor
  mapping the coalgebra $\xi : X \rightarrow GX$ to
  $X \stackrel{\xi}{\rightarrow} GX \stackrel{\lambda'}{\rightarrow} MFPX$.
  $\lambda'$ is the twisted distributivity law.
  \begin{align}\label{eq:lifts}
	  &\begin{tikzcd}[ampersand replacement=\&, column sep=2cm]
        FA \arrow{dd}{\alpha}
           \arrow{r}{Ff}
          \&
        FPX \arrow{d}{\lambda}
            \arrow[bend left=50]{dd}{\hat{P}\xi}
         \\
          \&
        PGX \arrow{d}{P\xi}
         \\
        A \arrow{r}{f}
          \&
        PX      
	  \end{tikzcd}
	  &&&
	  \begin{tikzcd}[ampersand replacement=\&, column sep=2cm]
        MFPX \arrow{r}{MFf}
          \&
        MFA 
         \\
        GX \arrow[swap]{u}{\lambda'}
          \&
         \\
        X \arrow[swap]{u}{\xi}
          \arrow{r}{f'}
          \arrow[bend left=50]{uu}{\Lambda\xi}
          \&
        MA \arrow[swap]{uu}{M\alpha}     
	  \end{tikzcd}
  \end{align}
\item[(c)] If $\Cal{T}$ is a regular category, and $F : \Cal{T} \rightarrow \Cal{T}$
  preserves reflexive coequalizers, then ${}_{F}\Cal{T}$ is a regular category.
  Every $F$-algebra homomorphism $\alpha \stackrel{f}{\rightarrow} \hat{P}\xi${}
  has a regular epi-mono factorization.
\item[(d)] If $\Cal{S}^{op}$ is a regular category, and $MFP$ preserves
  weak pull backs, then every twisted coalgebra homomorphism
  $f' : \Lambda\xi\rightarrow M\alpha$ has a regular epi-mono factorization,
  which induces a coalgebra $\ell : L \rightarrow MFPL$ as the image
  of $\Lambda\xi$  
  \begin{align}\label{eq:epimono}
	  \begin{tikzcd}[ampersand replacement=\&, column sep=2cm]
        MFPX
           \arrow[two heads]{r}{MFP\,e}
          \&
        MFPL
          \arrow[tail]{r}{MF\,m'}  
          \&
        MFA  
         \\
        GX
           \arrow{u}{\lambda'}
          \& \&
          \\
        X  \arrow[two heads]{r}{e}
           \arrow{u}{\xi}
           \arrow[bend left=50]{uu}{\Lambda\xi}
          \&
        L \arrow[tail]{r}{m}
          \arrow[dashed]{uu}{\ell}
          \&
        MA
          \arrow[swap]{uu}{M\alpha}      
	  \end{tikzcd}
  \end{align}  
  \end{itemize}	
\end{thm}

\paragraph{Examples.} Testing a standard coalgebraic presentation of
automata (e.g.~\cite{RuttenJ:Universal}) by linear process formulas
yields regular languages as the behaviors in $L$. The $MFP$-coalgebra
structure is induced by the language derivative. Testing the same
family of automata by a process logic with branching yields the
behaviors modulo the various bisimulation concepts from the branching
time spectrum~\cite{vanGlabbeek-I,vanGlabbeek-II}. Testing a coalgebraic
version of pushdown automata or Turing machines by linear process
formulas yields context-free languages, resp. recursively enumerable
languages as the behaviors in $L$, again with the language derivative
as the coalgebra structure. The details are in the Appendix
of~\cite{Pavlovic:Mislove:Worrell:2006a}\footnote{As mentioned earlier,
the Appendix was omitted from the published volume. The original version
of the paper is available online.}.
The approach also extends to nondeterministic and probabilistic systems.
The reason why testing coalgebras by algebras turns out to be both so
versatile, and also simpler than most of the purely coalgebraic
approaches to process semantics, is that it is, in a sense, just a
categorical formalization of the familiar practice of defining
computational models by specifying state machines on one hand and a
language acceptance relation on the other. It should be noted, though,
that, e.g., the purely coalgebraic approaches to trace semantics along
the lines of~\cite{Hasuo-Jacobs-Sokolova} are certainly more
appropriate for some purposes than our testing trace semantics, while
the non-coalgebraic approaches to computability and complexity, e.g.
through a monoidal computer~\cite{PavlovicD:IC12}, are more convenient
for other purposes than the coalgebraic Turing machines
of~\cite{Pavlovic:Mislove:Worrell:2006a}. And finally, while all
examples discussed so far are concerned with computation, our goal in
the present paper is to perhaps make a further step following the
ideas of~\cite{PavlovicD:LAPL}, and to provide a coalgebraic
reconstruction of some concepts from mathematical analysis.

\paragraph{Vector calculus by testing.} It turns out that both vector
fields and differential forms can be reconstructed using a simple
semantic connection where $\Cal{S}= \Cal{T}$ is the category of
manifolds viewed within a topos, and $MX=PX =\RR^X$, i.e. the dualizing
object, is an object of reals. In this setting, vector fields (or
dynamical systems) will be viewed as processes, presented as coalgebras,
whereas differential forms can be viewed as tests, and they can be
presented as algebras.

%=======================================================================
% From gros topos to manifolds
%-----------------------------------------------------------------------
% !TEX root = manifolds.tex
%=======================================================================
\section{Elements of vector calculus}\label{sec:manifolds}
%=======================================================================
\subsection{Differentials}
One way to establish that two functions behave similarly at a point is
to define their \emph{approximation}\/ relation.
\begin{defn}\label{def-approximate}
Let $X$ and $Y$ be metric spaces with the distance functions
$\dist_X: X\times X\to \RR_+$ and $\dist_Y: Y\times Y\to \RR_+$, and
let $f,g:X\to Y$ be continuous functions. We say that $f$ and $g$
\emph{approximate}\/ each other at $c\in X$ whenever the distance
between $f$ and $g$ vanishes at $c$, i.e.
\begin{eqnarray}\label{eq-approximate}
f\underset c \sim g & \iff & \lim_{x\to c}\frac{\dist_Y \left(f(x), g(x)\right)}{\dist_X(x,c)} = 0
\end{eqnarray}
\end{defn}
For convenience, we list the axioms for metrics in the Appendix.
Differential calculus really begins when $X$ and $Y$ are not just
metric spaces, but also vector spaces, so that the question of
\emph{linear}\/ approximation can be asked. To tie the metric and the
linear structure together, it is usually required that the conditions
\[
\dist(x+u, y+u) = \dist(x,y) \qquad \qquad \dist(rx, ry) = |r|\cdot \dist(x,y)
\] 
hold for all vectors $u$ and all scalars $r$. Such a distance function
on a metric space $E$ can then equivalently be specified using a norm
$\Vert - \Vert : E\to \RR_+$, which is related with $\dist: E\times E\to \RR_+$ by
\[\Vert x \Vert = \dist(x,0) \qquad\qquad \dist(x,y) = \Vert x-y \Vert \]
The axioms for the norm are also in the Appendix. This brings us into
the realm of \emph{normed vector spaces}. They are a staple of
functional analysis, at least when they are complete under the induced
topology: complete normed vector spaces are known as \emph{Banach}\/ spaces.
Since our goal here is to reconstruct the concepts of vector calculus
from the minimal structural assumptions, we ignore the completeness
requirement for the moment. Normed vector spaces, and their conceptual
ancestors and origins were studied thoroughly and exhaustively\footnote{Grothendieck
subsequently left the area, explaining to Malgrange: 'There is nothing
more to do, the subject is dead.' \cite[p.~1045]{FEA}} in Grothendieck's
thesis~\cite{GrothendieckA:TVS}. The norm and the linear structure
suffice for defining the notion of differential. 

\begin{defn}\label{def-differential}
Let $X$ and $Y$ be normed vector spaces. The norm structure determines
the set $\Cc(X,Y)$ of continuous maps, whereas the vector space
structure determines the set $\Lin(X,Y)$ of linear maps. The
\emph{differential} of $f\in \Cc(X,Y)$ is the partial operation
	\begin{align*}
		\bfDk f : X \rightharpoonup \Lin(X,Y)
	\end{align*}
such that the linear operator $\bfDk f(c)$, whenever it exists,
approximates the difference $f(x)-f(c)$, in the sense
\begin{eqnarray}
f(x) - f(c) &\underset c \sim & \bfDk f(c)(x-c)
\end{eqnarray}
where $x$ is a variable and $c$ a constant, i.e. by unfolding~\eqref{eq-approximate}
\begin{align*}
		\lim_{x\rightarrow c} \frac{%
		  \Vert f(x)-f(c)- \bfDk f(c)\cdot (x-c)\Vert}{%
		  \Vert x-c\Vert} = 0.
	\end{align*}
The function $f$ is \emph{differentiable} if $\bfDk f$ is defined for
all $c\in X$. The set of differentiable functions is written $\Cc^1(X,Y)$.
\end{defn}

The differentials can in fact be viewed as the linearizations of the
equivalence classes modulo the approximation relation~\eqref{eq-approximate}.
The following lemma spells this out formally. The proof is a standard
exercise with the norm.

\begin{lem}\label{lem-equiv}
Let $f, g\in \Cc^1(X,Y)$ be differentiable functions between normed
vector spaces. Then $f\underset c \sim g$ holds if and only if
$\bfDk f(c) = \bfDk g(c)$ and  $f(c)=g(c)$.
\end{lem}

Whenever $\Lin(X,Y)$ is metrizable, and thus a normed vector space,
then the differential construction can be iterated, i.e. applied to
the differential itself whenever it is continuous and thus
$\bfDk f \in \Cc\left(X, \Lin(X,Y)\right)$, leading to
	\begin{align*}
		\bfDk^2 f = \bfDk \left(\bfDk f\right)\  : X \rightharpoonup \Lin\left(X, \Lin(X,Y)\right) \cong \Lin\left(X\otimes X, Y\right)
	\end{align*}
Note that $\Lin(X,Y)$ is not metrizable in general, but the peculiar
situations when it is not have been characterized in Mackey's work, and
clearly spelled out in~\cite[Ch.~3]{GrothendieckA:TVS}. In general,
whenever $\bfDk^i f \in \Cc\left(X, \Lin(X^{\otimes n},Y)\right)$ for
all $i = 0, 1, \ldots k$ we define
\begin{align*}
		\bfDk^{k+1} f = \bfDk \left(\bfDk^k f\right)\  : X \rightharpoonup  \Lin\left(X^{\otimes k+1}, Y\right)
	\end{align*}
and write $f \in \Cc^k(X,Y)$ when $f$ is $k$ times differentiable. The
maps which are $k$ times differentiable for all $k$ are in 
\bear
\Cc^\infty(X,Y) & = & \bigcap_{k=0}^\infty \Cc^k(X,Y)
\eear
The convention is that $\bfDk^0 f = f$ and $\Cc^0(X,Y) = \Cc(X,Y)$, and
of course $\Lin(X^{\otimes 0}, Y) = \Lin(\RR, Y) = Y$.

\begin{defn}\label{def-NVS}
Denote by $\Cat{NVS}_k$ the category where the objects are normed
vector spaces and the morphisms are $k$ times differentiable maps,
i.e. $\Cat{NVS}_k(X,Y)$ $= \Cc^k(X,Y)$. The \emph{smooth}\/ morphisms
are in $\Cat{NVS}_\infty(X,Y) = \Cc^\infty(X,Y)$.
\end{defn}

\subsection{Manifolds}
Vector calculus is built upon the differentials as its main conceptual
tool, and largely geared towards computing the linear approximations
of differentiable processes. Since nonlinear processes can only be
linearly approximated~\emph{locally}, the spaces where processes are
presented do not have a global linear structure, but only local linear
structures. This means that every point has a neighborhood homeomorphic
with a normed vector space. This is the idea of \emph{manifold}, going
back to Riemann and Poincar\'e, worked out in modern algebraic
topology~\cite{DoldA:Lectures}.

\begin{defn}
A \emph{manifold} is a topological space $M$ whose every point has a
neighborhood homeomorphic with convex open subset of a normed vector
space. More precisely, a manifold consists of
\begin{itemize}
\item the underlying topological space $M$,
\item a normed vector space $E$, called \emph{model space},
\item an open cover $\UUU = \{U_i\ |\ i\in I\}$ of $M$, i.e. $\bigcup_{i\in I} U_i = M$,
\item the open embeddings $U_i \stackrel{\varphi_i} \mono E$ for all $i$,
      with convex images, called \emph{charts}.
\end{itemize}
The set 
$
\AAA  =  \left\{U_i \stackrel{\varphi_i} \mono E\ |\ i\in I\right\}
$
is called the \emph{atlas} of the manifold $M$. 
\end{defn}

\paragraph{Remarks.} Note that the definition did not mention the
differential structure. It will be lifted from normed vector spaces
to manifolds in a moment. The point of the current purely
\emph{topological}\/ definition is that at each point $x\in M$ the
chart $U \stackrel \varphi \mono E$ with $U\ni x$ induces a local
linear structure; or more precisely, a local \emph{convex}\/ structure,
since only the convex combinations of the points in the neighborhood
$U$ can be formed\footnote{Other reasons and consequences of the
convexity requirement on charts are spelled out
in~\cite[Ch.~2]{GrothendieckA:TVS}.}. This suffices for defining
differentials. In most applications, each vector space $E$ is assumed
to be given with a basis, and the charts thus induce the local
coordinate systems within $M$. The spaces $E$ are usually even
Euclidean, and thus presented as $\RR^n$. For a moment, we stick
with abstract normed vector spaces not so much for the sake of
generality, as to emphasize what structure is really needed to get
differential calculus going. Since normed vector spaces suffice for
defining the differential, the charts induce the local
differential structures. The only problem is that over the open sets,
where the different charts may intersect, different differentials
may be induced. The next definition preempts that.

\begin{defn}
A manifold $M$ is \emph{$k$ times differentiable} if the charts
induce the same $k$-tuple differentials over their intersections. More
precisely, for any two charts $U_0 \stackrel{\varphi_0} \mono E$ and
$U_1 \stackrel{\varphi_1} \mono E$ in the atlas of $M$ it is
required that the homeomorphism
\[
\phi(U_0\cap U_1) \tto{\varphi_0^{-1}} U_0\cap U_1 \tto {\varphi_1} \varphi_1(U_0\cap U_1)
\]
is $k$ times differentiable, i.e. it is a $\Cc^k$-diffeomorphism of the open convex images $\phi_0(U_0\cap U_1)$ and $\varphi_1(U_0\cap U_1)$ of $U_0\cap U_1$. A manifold is
\emph{smooth} if it is differentiable for all positive integers $k$, having $\Cc^{\infty}$-diffeomorphisms.
\end{defn}

A manifold is thus a topological space $M$ with a model space $E$ and
an atlas $\AAA$ that lifts the differentials from $E$ to $M$. A morphism
between $k$ times differentiable manifolds should preserve the
$k$-tuple differentials. To align the continuity with the
differentiability of the manifold morphisms, we need to slightly
refine the notion of atlas.

\begin{defn}
The atlas $\AAA$ of a manifold $M$ is \emph{saturated} if for every
chart $U\tto \varphi E$ in $\AAA$ and every open subset $W$ of $U$
the restriction $W\inclusion U\tto \varphi E$ is also an element of $\AAA$.
\end{defn} 

It is easy to see that every atlas has a unique saturation. A saturated
atlas allows finding sufficiently small charts.

\begin{defn}
Let $M$ and $N$ be $k$ times differentiable manifolds with saturated
atlases. A continuous map $f:M\to N$ is a \emph{$k$ times differentiable}
morphism if for every $x\in M$ and the neighborhoods $V\in f(x)$ and
$U\ni x$ such that $U\subseteq f^{-1}(V)$, with the charts
$U\stackrel\varphi \mono E$ and $V\stackrel \psi \mono F$ the continuous map
\[
f_{\psi\varphi} \,:\, \varphi(U) \tto{\varphi^{-1}} U \tto f V \tto{\psi} \psi(V) 
\] 
is $k$ times differentiable, i.e.
$f_{\psi\varphi} = \psi\circ f\circ \varphi^{-1} \in \Cc^k\left(\varphi(U), \psi(V)\right)$.{}
This is illustrated in Fig.~1.

\begin{figure}[h]
 \[ \begin{tikzpicture}
	  \draw (0,0) -- (1,2) -- (4,2) -- (3,0) -- cycle;
	  \node at (0.25,1) {$E$};
	  \draw[rounded corners=4pt] (1,0.5) -- (1.5,1.5) -- (3,1.5) -- (2.5,0.5) -- cycle;
	  \node at (3.25,1.7) {$\varphi(U)$};
	  \fill[black!80] (1.5,1) circle (0.5ex) node[xshift=4ex] {$\varphi(x)$};
     \draw[rounded corners=10pt]  (0,6) -- (2,7) -- (4,6) --
        (4,3) -- (2,4) -- (0,3) -- cycle;
	  \node at (3.5,6.5) {$M$};
	  \node at (3,5.25) {$U$}; 
	  \draw (2,5) ellipse (25pt and 20pt);
	  \fill[black!80] (1.75,5.25) circle (0.5ex) node[xshift=2ex]{$x$};
	  \draw[->,thick] (1.5,5) ..controls +(-0.5,-1) and +(-1,2) .. node[pos=0.7,xshift=-2ex]{$\cong$} node[pos=0.7,xshift=2ex]{$\varphi$} (1.25,1.25);
	  \draw (5,0) -- (6,2) -- (9,2) -- (8,0) -- cycle;
	  \node at (8.75,1) {$F$};
  	  \draw[rounded corners=4pt] (6,0.5) -- (6.5,1.5) -- (8.5,1.5) -- (8,0.5) -- cycle;
	  \node at (8.25,1.7) {$\psi(V)$};
	  \fill[black!80] (6.5,1) circle (0.5ex) node[xshift=5ex] {$\psi(f(x))$};
     \draw[rounded corners=10pt]  (5,7) -- (7,6) -- (9,7) --
        (9,4) -- (7,3) -- (5,4) -- cycle;	     
	  \node at (6.5,6.5) {$N$};
	  \node at (8,5.25) {$V$};
	  \draw (7,5) ellipse (25pt and 20pt);
	  	  \fill[black!80] (6.5,5.25) circle (0.5ex) node[xshift=3.5ex]{$f(x)$};   
	  \draw[->,thick] (6.75,5) ..controls +(0.5,-1) and +(1,1.5) .. node[pos=0.7,xshift=-2ex]{$\cong$} node[pos=0.7,xshift=2ex]{$\psi$} (6.75,1.25);
	  \draw[->,thick] (3,7) ..controls +(1,0.5) and +(-1,0.5) .. node[yshift=2ex]{$f$} (6,7);
 	  \draw[->,thick] (2.8,4.5) ..controls +(1,-0.5) and +(-1,-0.5) .. node[yshift=2ex]{$f_{U}$} (6.2,4.5);
 	  \draw[->,thick] (2.8,1) ..controls +(1,-0.25) and +(-1,-0.25) .. node[yshift=2ex]{$f_{\psi\varphi}$} (6.2,1);
  \end{tikzpicture}
 \]
 \label{fig-Mmorphism}
  \caption{$k$-times differentiable map $f : M \rightarrow N$}
 \end{figure}
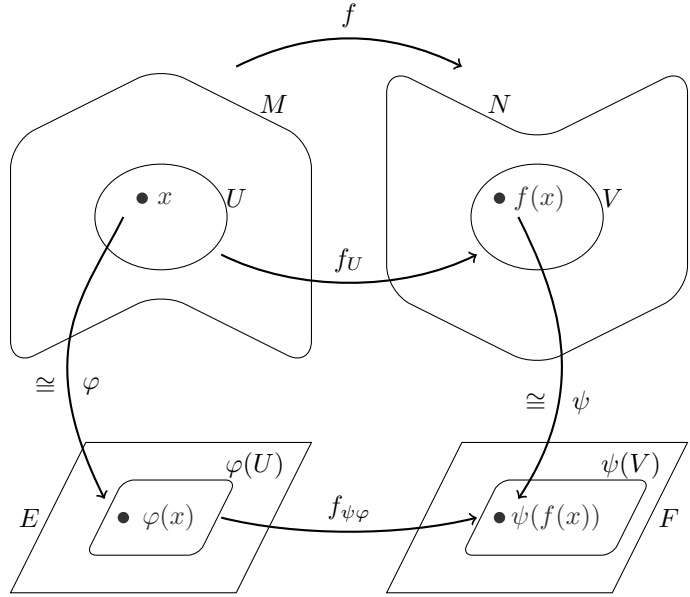
The \emph{smooth} morphisms are $k$ times differentiable for all $k$,
i.e. they are the elements of $\Cc^\infty(M,N) = \bigcap_{k=0}^\infty \Cc^k(M,N)$.

The category of $k$ times differentiable manifolds and their morphisms
is denoted $\Cat{Man}_k$. The category of smooth manifolds and smooth
morphisms is denoted $\Cat{Man}_\infty$.
\end{defn}

Modern vector calculus is largely built using differentiable manifolds~\cite{Abraham:Marsden:Ratiu:1988a}. The model spaces $E$ are usually assumed to be Banach spaces, i.e. normed vector spaces that are complete under the induced topology. However, the completeness plays no role in the definition of manifolds, or even of their tangents and vector fields. It is, of course, essential in solving differential equations and analyzing dynamical systems, but conspicuously unnecessary for describing them. We shall thus deviate from the bulk of literature where manifolds are viewed as continuously varying Banach spaces, and work with manifolds as continuously varying normed vector spaces. Although perhaps less familiar, this view simplifies our categorical treatment.

\subsection{Tangent vectors}
The idea of a \emph{tangent vector}\/ is that it is the linear
approximation of a curve at a given point. In vector calculus, the
geometric idea of linear approximation is captured by the concept of
the differential. If the curves in a normed vector space are viewed
as the differentiable maps from $\RR$, then we know from
Lemma~\ref{lem-equiv} that two curves $\varsigma, \chi \in \Cc^1(\RR,E)${}
through the same point $e = \varsigma(0) = \chi(0)$ have the same
differential if and only if $\varsigma \underset 0 \sim \chi$. The
tangent vectors through a point $e$ of a normed vector space $E$ can
thus be viewed as the equivalence classes modulo $\underset 0 \sim$.
Lifting the differentials from normed vector spaces to manifolds, the
tangent vectors on manifolds are defined in the same way, as illustrated
on Fig.~2.

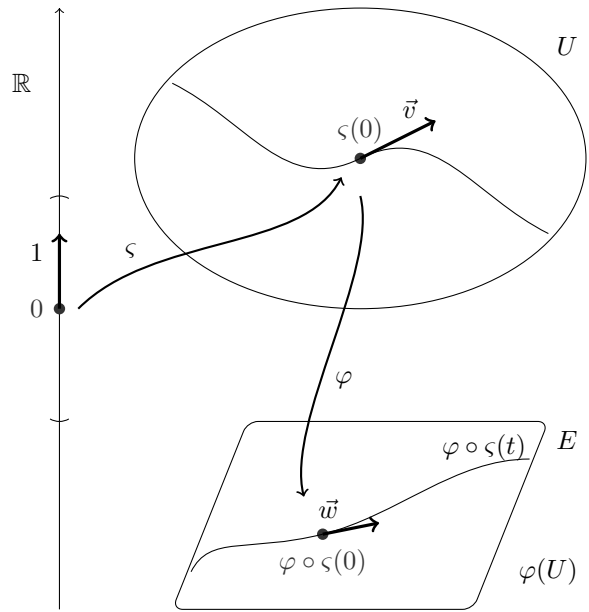
\begin{figure}[h]
\[ \begin{tikzpicture}
	  \draw[->] (0.5,0) to (0.5,8);
	  \node at (0,7) {$\RR$};
	  \node at (0,5.5) {};
	  \fill[black!80] (0.5,4) circle (0.5ex) node[xshift=-2ex]{$0$};
	  \draw[->,very thick] (0.5,4) to node[above,xshift=-2ex]{$1$} (0.5,5);
	  \draw (0.5,5.5) arc (90:60:0.25cm);
	  \draw (0.5,5.5) arc (90:120:0.25cm);
	  \draw (0.5,2.5) arc (270:300:0.25cm);
	  \draw (0.5,2.5) arc (270:240:0.25cm);
	  \draw[rounded corners=4pt] (2,0) -- (3,2.5) -- (7,2.5) -- (6,0) -- cycle;
	  \node at (7,0.5) {$\varphi(U)$};
	  \fill[black!80] (4,1) circle (0.5ex) node[yshift=-2.5ex]{$\varphi\circ\varsigma(0)$};
	  \draw (2.25,0.5) .. controls +(0.25,0.5) and +(-1,-0.25) .. (4,1);
	  \draw (4,1) .. controls +(1,0.25) and +(-1,0) .. (6.75,2) node[xshift=-4ex,yshift=1ex]{$\varphi\circ\varsigma(t)$};
	  \draw[->,very thick] (4,1) to node[above,xshift=-2ex]{$\vec w$} (4.75,1.15);
	  \draw (4.5,6) ellipse (3cm and 2cm);
	  \node at (7.25,7.5) {$U$};
	  \fill[black!80] (4.5,6) circle (0.5ex) node[yshift=2.5ex]{$\varsigma(0)$};
	  \draw (2.5,0.5) (2,7) .. controls +(1,-0.5) and +(-1,-0.5) .. (4.5,6);
	  \draw (4.5,6) .. controls +(1,0.5) and +(-1,0.5) .. (7,5);
	  \draw[->,very thick] (4.5,6) to node[above,xshift=1ex,yshift=1ex]{$\vec v$} (5.5,6.5);
  	  \draw[->,thick] (0.75,4) ..controls +(1,1) and +(-0.5,-1) .. node[pos=0.2,yshift=2ex]{$\varsigma$} (4.25,5.75);
  	  \draw[->,thick] (4.5,5.5) ..controls +(0.25,-1) and +(-0.25,1) .. node[pos=0.6,xshift=2ex]{$\varphi$} (3.75,1.5);
%BF
     \node at (7.25,2.25) {$E$};
  \end{tikzpicture}
\]    
  \caption{A tangent vector on the manifold $M$ can be obtained in the form $\vec v = \bfDk \varsigma(0)\cdot 1$,  where $\varsigma:\RR \to M$ is a suitable curve through the tangent point. By definition, the differential $\bfDk \varsigma(0)$ in $M$ is the differential  $ \bfDk(\varphi\circ \varsigma)(0)$ in $E$, where $U\stackrel \varphi \mono E$ is a chart at $\varsigma(0)$.}
  \end{figure}
  
\begin{defn}
	A $k$ times \emph{differentiable curve} or \emph{path} on a manifold
	$M$ is a map $\varsigma\in C^{k}(\RR,M)$.
\end{defn}

\begin{defn}\label{def-bundle}
A \emph{tangent vector} (or simply \emph{tangent}) over a differentiable
manifold $M$ is the equivalence class of paths in $M$ modulo the
approximation relation. That is the set of all tangents over $M$ is
collected in its \emph{tangent bundle}
\bea
\T M & = & \Cc^1(\RR,M) / \underset 0 \sim
\eea
The equivalence classes of curves $[\varsigma]_\sim \in \T M$ are often
called\/ \emph{linearizations} of their elements.
\end{defn}

The point of collecting the tangents over a manifold $M$ into a bundle
is that the tangent bundle $\T M$ is a manifold again, and that dynamical
systems can be presented as suitable manifold morphisms $M\to \T M$.
While this provides a useful high level view of dynamical systems, the
direct implementation of this idea requires working with the low level
representations of the tangents, and unfolding charts and atlases over
and over again. To give the reader a flavor of this implementation, we
sketch it in the rest of this section. In the next section, we propose
a high level view of tangent bundles, as behaviors recognized through
testing.

\subsubsection{The tangent bundle functor $\T$}
By definition, the elements of the tangent bundles are the equivalence
classes of curves $[\varsigma]_\sim$, where $\varsigma \in \Cc^1(\RR,M)${}
is a representative. By mapping each tangent $[\varsigma]_\sim$ to the
point where it touches the manifold, we get the projection
\bea\label{eq-projection}
\varepsilon\ :\ \T M & \to & M\\[0pt]    
[\varsigma]_{\sim} &\mapsto & \varsigma(0) \notag
\eea
which is well defined by Lemma~\ref{lem-equiv}. We need to lift to
$\T M$ not just the topology, but also the atlas from $M$. Towards
that goal, let us first consider a very special case. 

\begin{lem}\label{def:trivialization}
Consider an open set $U$ in a normed vector space $E$ as a manifold,
with the obvious chart. Then $\T U \cong U\times E$.
\end{lem}

\begin{proof}[Proof sketch]
Since the linearization of the path $\varsigma\in C^{1}(\RR,E)$ at
$\varsigma(0)$ is generally given by evaluating the differential
$\bfDk\varsigma : \RR \rightarrow \Lin(\RR,E) = E$ at $0$ to get
$\bfDk \varsigma (0)\in E$, in this particular case we have
\bea\label{proof-map}
	\T U & \to & U\times E\\[0pt]
[\varsigma]_\sim &\mapsto & < \varsigma(0),\bfDk \varsigma(0)\cdot 1 > \notag
\eea
where $1$ is the unit vector at $0$ in $\RR$. 
\end{proof}

\begin{prop}\label{prop:Tis manifold}
	$\T M$ carries a canonical manifold structure which restricts to its open submanifolds
	\begin{equation}\label{diag-ext}
\begin{tikzcd}[ampersand replacement=\&,column sep=large]
      \T U \arrow[tail]{r}\arrow[swap]{d}{\varepsilon} \& \T M \arrow{d}{\varepsilon}
      \\
      U\arrow[hookrightarrow]{r} \& M
    \end{tikzcd}
\end{equation}
If $M$ is a $k$ times differentiable manifold, then $\T M$ is $k-1$ times differentiable.
\end{prop}

\begin{proof}[Proof sketch]
Suppose that the manifold $M$ is given with the atlas 	
\bear
\calA  & = & \left\{U_i \tto{\varphi_i} E\ |\ i\in I\right\}
\eear 
We claim  that 
	 \bear 
	 \T \calA & = & 
	\left\{U_i\times E \tto{\varphi_i\times E} E\times E\ |\ i\in I \right\}\eear
is the atlas for $\T M$ which restricts to $M$'s submanifolds as
claimed. For each linearization $[\varsigma]_\sim \in \T M$, the atlas
$\calA$ gives $\la U,\varphi\ra \in \calA(\varsigma(0))$. In the
neighborhood of $q=\varsigma(0)$ the paths $\varsigma$ and
$\varphi\circ\varsigma$ can be used interchangeably. But
$\delta =\varphi\circ \varsigma : \RR \rightarrow \varphi(U)$
is given by $\delta(t) = \varphi(\varsigma(0)) + t \bfDk(\varphi\circ\varsigma)(0)\cdot e$
and $\delta$ is completely determined by $\la \varphi(\varsigma(0)) , v\ra$,{}
with $v=\bfDk(\varphi\circ\varsigma)(0)\cdot e = D\cdot e \in E$. Thus
the cylinder over $\phi(U)$ is a neighborhood in $\T M$, i.e.
$\phi(U)\times E \in \calN\left([\varsigma]_\sim\right)$. 
\end{proof}

Proposition~\ref{prop:Tis manifold} determines the object part of the
functor $\T :\Cat{Man}_{k+1}\to \Cat{Man}_k$. The arrow part is defined
along the same lines.

\begin{defn}\label{def:reduceclass}
Define the category of \emph{basic} manifolds with $k$ times
differentiable maps $\Cat{BMan}_k$ as follows:
\begin{itemize}
\item the objects are pairs $<U,E>$ where $E$ is a normed vector space,
      and $U\subseteq E$ is an open convex subset;
\item the morphisms $<U,E> \tto f <V,F>$ are the $k$ times differentiable
      maps $f\in \Cc^{k}(E,F)$ such that $f(U)\subseteq V$.
\end{itemize}
The category $\Cat{BMan}_\infty$ has the same objects, but the morphisms
are smooth, i.e. $\Cat{BMan}_\infty (M,N) = \bigcap_{k=0}^\infty \Cat{BMan}_k (M,N)$.
\end{defn}

\begin{prop}\label{prop-BMan}
The following mappings define a functor
$\T : \Cat{BMan}_{k+1} \rightarrow \Cat{BMan}_{k}$
\begin{align*}
	&\begin{tikzcd}[ampersand replacement=\&,column sep=large]
      U\subseteq E \arrow[|->]{d}
      \\
      U\times E
    \end{tikzcd},
	&
	\begin{tikzcd}[ampersand replacement=\&,column sep=3.2cm]
		U\subseteq E \arrow{r}[name=U,below]{}{f}  
		    \&
		V\subseteq F 
		  \\
		U\times E \arrow{r}[name=D,above]{}[swap]{\T f=\la f\circ\pi_{0}, \bfDk f(\pi_{0})\cdot\pi_{1}\ra}
		   \&  
	   V\times F
	   \arrow[|->,to path= (U) -- (D)]{}; 
   \end{tikzcd}
\end{align*}
Restricted to the smooth morphisms in $\Cat{Ban}_{\infty}$, the above
construction gives the endofunctor
$\T : \Cat{BMan}_{\infty} \rightarrow \Cat{BMan}_{\infty}$.
\end{prop}

\begin{proof}
We need to check that the arrow part of the construction is functorial.
The fact that $\T ({\rm id}) = {\rm id}$ follows from
Lemma~\ref{lem-linear} below. By unfolding the definition, which says that 
\bear
\T f (q,v) & = & \left<f(q), \bfDk f(0) \cdot v\right>
\eear
and using the chain rule (proved in the usual way) for the differential
from Def.~\ref{def-differential}, we prove the claim: 
\bear
\T (f\circ g)(q,v) & = & \big<f(g(q)), \bfDk (f\circ g)(q) \cdot v\big>\\
	 & = &  \big<f(g(q)), \bfDk f (g(q)) \cdot \bfDk g(q) \cdot  v\big>\\
	 & = &  \T f \left(g(q), \bfDk g(q) \cdot  v\right) \\
	 & = &  \left(\T f \circ \T g\right) (q,v)
	 \eear
\end{proof}

\paragraph{Remarks.} This exhibits the fact that $\T (f\circ g) = \T (f) \circ \T (g)$
is locally equivalent to the chain rule. The projection
\bear
\Cat{BMan}_k & \to & \Cat{NVS}_k\\
<U,E> & \mapsto & E
\eear
makes each $\Cat{BMan}_k$ into a fibered category over $\Cat{NVS}_k$.
It is important to note that $\T$ is not a lifting of the squaring functor 
\bear
S\ :\ \Cat{BMan}_k & \to & \Cat{BMan}_k\\
X &\mapsto & X\times X
\eear
although the object part of $\T$ defined above, mapping $U\subseteq E$
to $U\times E$, might suggest that it is. But the arrow part is obviously
different: while the squaring functor maps $f\mapsto f\times f$, the
second component of $\T f$ is not $f$, but its linear approximation.
\begin{lem}\label{lem-linear}
$\T f = f\times f$ if and only if $f$ is linear.
\end{lem}

\begin{prop}\label{prop-extension}
The functors $\T : \Cat{BMan}_{k+1} \to \Cat{BMan}_k$ have unique
extensions to $\T : \Cat{Man}_{k+1} \to \Cat{Man}_k$ along the inclusions
of basic manifolds into differentiable manifolds in general:
\[	\begin{tikzcd}
\Cat{BMan}_{k+1} \arrow{r}{\T} \arrow[hookrightarrow]{d}& 
\Cat{BMan}_{k}\arrow[hookrightarrow]{d} \\
\Cat{Man}_{k+1} \arrow[dashed]{r}{\T}& 
\Cat{Man}_{k}
	\end{tikzcd}
\]
\end{prop}

\begin{proof}
The object part of the extension maps the $k+1$ times differentiable
manifolds $M$ to the tangent bundles $\T M$ as described in
Def.~\ref{def-bundle}, with the manifold structure described in
Prop.~\ref{prop:Tis manifold}. Towards the arrow part, for any
$f\in \Cat{Man}_{k+1}(M,N)$ and any tangent $[\varsigma]_\sim \in \T M$
represented by a differentiable curve $\varsigma:\RR \to M$, define
\bea\label{eq-Tf}
\T f\left([\varsigma]_\sim\right) & = & \left[f\circ \varsigma\right]_\sim
\eea
The fact that this definition is sound, i.e. that $\varsigma \sim \varsigma'${}
over $M$ implies $f\circ \varsigma \sim f\circ \varsigma'$ over $N$
follows from Lemma~\ref{lem-equiv}, using the chain rule. The fact
that $\T f$ is a $k$ times differentiable manifold homomorphism,
i.e. in $\Cat{Man}_k\left(\T M, \T N\right)$, follows from the
fact that it restricts to $\T f_U \in \Cat{BMan}_k\left(\T U, \T V\right)${}
for all corresponding charts $U\stackrel \varphi \mono E$ of $M$ and
$V\stackrel\psi \mono F$ of $N$ and the restriction
$f_U\in \Cat{BMan}_k\left(U ,V\right)$ constructed as in Figure~1.
They make the top trapezoid in the following diagram commutative.
\begin{equation}\label{diag-look}
		\begin{tikzcd}[ampersand replacement=\&, column sep=2cm]
			M \arrow{rrr}{f}
			  \& \& \& N \\
			  \& U \arrow[hook]{lu}{\varphi}
			       \arrow{r}{f_{U}}
			       \& V \arrow[hookrightarrow]{ru}{\psi} \& \\
			  \& U\times E
			       \arrow{u}{\varepsilon} 
			       \arrow[tail]{dl}{\varphi\times E}
			       \arrow{r}{\la f\pi_{0},\bfDk f\cdot \pi_{1}\ra}
			       \& V\times F 
			          \arrow{u}{\varepsilon}
			          \arrow[tail]{rd}{\psi\times F}
			          \& \\
			\T M \arrow{uuu}{\varepsilon}
			   \arrow{rrr}{\T f}
			   \& \& \& \T N \arrow{uuu}{\varepsilon}     
		\end{tikzcd}
\end{equation}
The trapezoids on the sides commute as in~\eqref{diag-ext}, whereas
the square in the middle commutes by Lemma~\ref{lem:epsilon}, proved
below. The bottom trapezoid then shows how $\T f$, defined
in~\eqref{eq-Tf}, restricts to a mapping on the charts as defined in
Prop.~\ref{prop:Tis manifold}. Tracking these restrictions around
Figure~1 again, we see that $\T f$ is locally glued from the
differentials $\la f\pi_{0},\bfDk f\cdot \pi_{1}\ra$. If $f$ is
$k+1$ times differentiable, then $\bfDk f$ is $k$ times
differentiable, and thus each restriction of $\T f$ to a submanifold
$\T U = U\times E$ of $\T M$ is $k$ times differentiable. Therefore,
$\T f$ must be $k$ times differentiable itself.
\end{proof}

\begin{lem}\label{lem:epsilon} The maps $\varepsilon_M: \T M \to M$
defined in~\eqref{eq-projection}, are natural when $M$ ranges over
the basic manifolds $U$. Thus $\varepsilon: \T \to {\rm Id}$ is a
natural transformation between the endofunctors on $\Cat{BMan}_\infty$.
\end{lem}
\begin{proof}
Unfolding the definitions
\[	\begin{tikzcd}[ampersand replacement=\&, column sep=2cm]
		U\times E \arrow{r}{\la f\pi_{0},\bfDk f\cdot \pi_{1}\ra}
		          \arrow[equals]{d}
		 \&
		V\times F \arrow[equals]{d}
		 \\
		\T U \arrow{r}{\T f}
		     \arrow{d}{\varepsilon}
		 \&
		\T V \arrow{d}{\varepsilon}
		  \\
		U  \arrow{r}{f}
		 \&
		V    
	\end{tikzcd}
\]
shows that projecting curves to their points
$\varepsilon_U\left([\gamma]_\sim\right) = \varsigma(0)$, when
restricted to basic manifolds $U$, boils down to the projections
$\varepsilon = \pi_0 : U\times E \to U$.
\end{proof}
Just another look at diagram~\eqref{diag-look} gives
\begin{cor}\label{corollary-counit}
The maps $\varepsilon_M: \T M \to M$ form a natural transformation
$\varepsilon: \T \to {\rm Id}$ between endofunctors on $\Cat{Man}_\infty$.
\end{cor}

Proceeding along the same lines, we could further construct the natural transformation $\delta : \T \to \T\T$, first for
$\T$ as an endofunctor on $\Cat{BMan}_\infty$, where
$\T U \cong U\times E$, and thus
\bea\label{eq-coevaluation}
\delta\ :\ \T U & \to & \T\T U\\[0pt]    
<u,e> &\mapsto & <u,e,e,e > \notag
\eea
and then extending this along~\eqref{diag-ext} to $\T$ as an endofunctor on $\Cat{Man}_\infty$. Remarkably, though, the natural transformation $\T \tto \delta \T\T$ is coassociative, but not counitary with respect to $\T\tto \varepsilon {\rm Id}$ and therefore does not make $\T$ into a comonad. The reason is why it is not counitary, studied in more detail in~\cite{jubin:2012a} is that all natural transformations between nontrivial powers of $\T$ must preserve the zeros, as the trivial splitting of $\T M \tto \varepsilon M$. Vector fields over a manifold $M$ can be viewed as the general splittings of $\T M \tto \varepsilon M$, usually called \emph{cross sections}; but they do not embed into $\T M \tto \delta \T \T M$ as the cofree one, and they are thus not coalgebras of a monad, but just a special kind of coalgebras of the functor $\T : \Cat{Man}_\infty\to \Cat{Man}_\infty$. 

The whole process of functorializing tangent bundles, studied in the literature at many levels of generality and abstraction, can be repeated for cotangent bundles, leading to a cotangent bundle functor $\TCS:\Cat{Man}_\infty \to \Cat{Man}_\infty$ with a similar structure. While a tangent vector over a manifold $M$ is a $\sim$-equivalence class of curves $\varsigma:\RR\to M$, a cotangent vector is a $\sim$-equivalence class of differentiable functions $f:M\to \RR$. Just like vector fields can be presented as coalgebras $M\to  \T M$ such that the composite $M\to \T M \tto \varepsilon M$ is the identity, differential forms can be presented as coalgebras $M\to \TCS M$ such that the composite $M\to  \TCS M \tto \varepsilon M$ is the identity. Vector analysis is thus developed by using them together, and the cotangent bundle $\TCS M$ is usually defined as the dual of the tangent bundle $\T M$. At a closer look, it turns out that both the tangent and the contangent vectors arise by
\emph{testing}, in the sense of Sec.~\ref{sec:generaltesting}, of the differentiable functions $f:M\to \RR$ along the curves
$\varsigma:\RR\to M$. This provides a high level view of both constructions, which we present in the next section.

% From gros topos to manifolds
%-----------------------------------------------------------------------
% !TEX root = manifolds.tex

%=======================================================================
\section{Tangent and cotangent testing}\label{sec:tangenttesting}
\subsection{Manifolds as sheaves}\label{subsec:manifoldsassheaves}
Let $M$ be a manifold with a saturated atlas $\AAA_M$. The fact that
the atlas is saturated means that it can be presented as a \emph{presheaf},
i.e. a functor mapping each open set $U\subseteq M$ to the set of
charts over it, i.e.
\bea\label{eq-AAAM}
\AAA_M\ :\ \OOO(M)^{op} & \to & \Cat{Set}\\
U & \mapsto & \{U\tto \varphi E\}\notag
\eea
where $\OOO(M)$ is the set of opens in $M$, viewed as a partial order
the under inclusion\footnote{A partial order like $(\OOO(M), \subseteq)${}
is a category with at most one morphism $U\to V$, which exists just
when $U\subseteq V$}. If the model space $E$ is given with some chosen
coordinates $E\hookrightarrow \RR^X$, then every $k$ times
differentiable chart can be viewed as an $X$-tuple of $k$ times
differentiable maps $U\to \RR$, since there is the correspondence  
\[\prooftree
U\tto \varphi E\hookrightarrow \RR^X
\justifies
U\times X \tto{\varphi \times X} E\times X\to \RR
\endprooftree
\]
By viewing each $k$ times differentiable chart as an $X$-tuple of $k$
times differentiable functions to $\RR$, we can present the atlas $\AAA$
as the subfunctor of the sheaf of $k$ times differentiable functions:
\bea
\Cc^k(M)\ :\ \OOO(M)^{op} & \to & \Cat{Set}\\
U & \mapsto & \Cc^k(U)\notag
\eea
This is spelled out in~\cite[Sec.~II.3]{MacLane:Moerdijk:1992a}. When
$M$ is a second countable Hausdorff space, then all of the structure of
the manifold $M$ can be reconstructed from a subsheaf $\AAA \hookrightarrow \Cc^k(M)$.{}
Giving a $k$ times differentiable manifold on $M$ is equivalent to
giving a subsheaf of $\AAA\hookrightarrow \Cc^k(M)$ with the property
that $U$ embeds into $\RR^{\AAA U}$ for every $U$. Note, though, that
the atlases of the manifolds presented like this are not only saturated,
but also \emph{sheafified}, i.e. they contain not only all subcharts of
their charts, but also all unions of their charts that can be consistently
glued together~\cite[2.4]{TennisonB:Sheaf}. But since every saturated
atlas has a sheafification~\cite[Sec.~III.5]{MacLane:Moerdijk:1992a},
and the atlases that lead in this way to the same associated sheaf are
just different presentations of diffeomorphic manifolds, this leads
to no loss of generality. Sheaves thus provide a uniform presentation
of manifolds, eliminating the irrelevant implementation details which
allow different atlases for diffeomorphic, and thus indistinguishable
manifolds. This is, of course, not an accident, since the concept of
manifold was one of the guiding ideas behind the concept of
sheaf~\cite{Leray}.

The main technical advantage of the sheaf presentation is that manifolds
can now be studied in the rich environment of Grothendieck's
\emph{toposes}\/~\cite{SGA4}. This is where the tangent and cotangent
bundle functors will emerge from testing. 

The idea of a topos of sheaves $\Cat{Sh}(X)$ over a topological space
$X$ is that it is a \emph{generalized view}\/ of the space $M$. From
one angle, this view is supported by presenting sheaves as
\emph{continuosly variable sets}~\cite{LawvereFW:variable}, i.e. the
contravariant functors $\OOO(X)^{op} \to\Cat{Set}$ satisfying the sheaf
requirement (\cite[2.1]{TennisonB:Sheaf} or~\cite[Sec.~II.1]{MacLane:Moerdijk:1992a}),
analogous to the above requirement that all consistent charts are
glued together. The category $\Cat{Sh}(X)$ is then viewed as a
reflective subcategory of the functor category $\Cat{Set}^{\OOO(X)^{op}}$.
From a different angle, sheaves over a space $X$ are equivalently
presented as \emph{\'etale spaces}, i.e. local homeomorphisms
$E\to X$ (\cite[Sec.~2.3]{TennisonB:Sheaf} or~\cite[Sec.~II.6]{MacLane:Moerdijk:1992a}).
The functor category $\Cat{Sh}(X) \hookrightarrow \Cat{Set}^{\OOO(X)^{op}}$
is thus equivalent with the category of \'etale spaces $\EEE/X\hookrightarrow \Cat{Esp}/X$,
where $\Cat{Esp}$ is the category of topological spaces. This latter
presentation brings the sheaves over different spaces $X$ together in
the category $\Cat{Top}$ whose objects are \'etale spaces $E\tto p X$,
while the morphisms are commutative squares formed by the parallel
pairs of continuous maps $\la e, b\ra$ that make the square commute
in the next diagram.
\begin{align}
  \textrm{Obj:~}
  \begin{tikzcd} 
      E \arrow{d}[swap]{p} \\
      X
  \end{tikzcd}
  &&&
  \textrm{Mor:~}
  \begin{tikzcd}[ampersand replacement=\&] % why is this replacement needed here? 
      E \arrow{d}[swap]{p} \arrow{r}{e} \& E' \arrow{d}{p'} \\
      X \arrow{r}{b} \& X'
  \end{tikzcd}
\end{align}
The category $\Cat{Top}$ can be construed of as a crude version of
Grothendieck's 'gros topos'~\cite[IV.2.5]{SGA4}, along the lines
of~\cite{LawvereFW:gros}. 

\begin{prop}
There is a faithful functor $\Cat{Man}_k \mono \Cat{Top}$ for every
$k$. It is full for $k=0$.
\end{prop}

\begin{proof}
Each manifold $M$ is presented as the \'etale space $E\to M$ corresponding
to the sheaf $\AAA_M$ from \eqref{eq-AAAM}. The fact that the manifold
morphisms are just the continuous maps that preserve the fibers follows
directly from the definition of $\AAA_M$ as tracking the charts over
$M$, and the construction of the points of $E\to M$ as the neighborhood
filters through these charts.
\end{proof}

\subsection{Paths and tests}
The upshot of the presentation of manifolds as sheaves is that the
category $\Cat{Man}$ can thus be viewed as a subcategory of
$\Cat{Top}$. Moreover, the forgetful functor $\Cat{Top}\rightarrow\Cat{Esp}$,
mapping each \'etale space $E \rightarrow X$ to its base space $X$, is
a fibration, whose fibers are equivalent to the 'petit toposes'
$\Cat{Sh}(X)$.  

Each of them contains the constant sheaf $\RR$ of real numbers. The
differentiable paths $\RR\to M$ and the differentiable tests $M\to \RR${}
can now be internalized as sheaves.
\begin{defn}\label{def:GF}
Given a $k$ times differentiable manifold $M$, let the presheaves
$\Sigma M, \Theta M: \calO M^{op} \rightarrow \Cat{Set}$ be defined by
\begin{eqnarray}\label{eq-F}
	  \Sigma M(U) &= &\Cat{Man}_{k}(\Yoneda U\times \AAA_M,\RR) \\
	  \Theta M(U) &= & \Cat{Man}_{k}(\Yoneda U\times \RR,\AAA_M) \label{eq-G}
\end{eqnarray}
Setting  $\G M = \Sigma M$ and  $\FS M = \F M = \Theta M$ determines the
object parts of the functors
  \[
 \G, \F \ :\ \Cat{Man}_k \to \Cat{Top} \qquad \mbox{and} \qquad  
 \FS\ : \  \Cat{Man}_k^{op} \to \Cat{Top}     \]
Given a $k$ times differentiable map $h\in \Cat{Man}_k(M,N)$, the arrow
parts will be defined as follows:
\begin{eqnarray}
         \G h(U)\ :\ \G M(U) & \rightarrow & \G N(U) \label{def:path} \\
  \varsigma & \mapsto & h\circ \varsigma \notag \\
          \FS h(U)\ :\ \FS N(U) & \rightarrow & \FS M(U) \label{def:test} \\
     t & \mapsto & t\circ <{\rm id}, h> \notag \\
         \F h(U)\ :\ \F M(U) & \rightarrow & \F N(U) \label{def:covartest}\\
    s & \mapsto &  \F h(U)_s(u,y) = 
    \bigvee_{\varepsilon\searrow 0} \bigwedge_{\Vert h(x)-y\Vert \leq \varepsilon} s(u,x) \notag 
\end{eqnarray}
where $\varsigma \in\Cat{Man}_k(\Yoneda U\times \RR, M)$ is a path in
$M$ whereas $t\in\Cat{Man}_k(\Yoneda U\times N, \RR)$ and
$s\in\Cat{Man}_k(\Yoneda U\times M, \RR)$ are tests over $x\in M$
and $y\in N$. 
  \end{defn}
\paragraph{Remark.} Since the inclusion
$\Cat{Sh}(X) \hookrightarrow \Cat{Set}^{\OOO(X)^{op}}$ has a left exact
left adjoint~\cite[Sec.~III.5]{MacLane:Moerdijk:1992a}\footnote{The topos
theory results that we use are contained in most topos theory textbooks.
The choice of the one that we keep referring to is entirely a matter
of personal habit, and the reader may wish to consult any of the other
presentations.} the exponents of $\Cat{Sh}(X)$ are created in
$\Cat{Set}^{\OOO(X)^{op}}$. For $A, B \in \Cat{Sh}(X)$, the exponent
is thus 
\bear
B^A (U) & =  & \Cat{Sh}(\Yoneda U \times A, B)
\eear
Comparing this with~\eqref{eq-F} and~\eqref{eq-G}, it is clear that
$\Theta M$ is a subsheaf of the sheaf $\RR^M$, whereas $\Sigma M$ is
a subsheaf of $M^\RR$, where we abuse notation and write $M$ for the
sheaf $\AAA_M$, representing the manifold $M$. 

\paragraph{Convention.} To simplify notation, when reasoning about
sheaves and their morphisms, we often leave the variation over the
open sets $U\subseteq M$ implicit. This is formally justified not only
for the global constructions, but also for the local constructions
that are expressible in the internal language of a
topos~\cite[Ch.~VI]{MacLane:Moerdijk:1992a}.

\begin{prop}
The inclusions
\[
\G M \hookrightarrow M^\RR \qquad \mbox{ and } \qquad 
     \FS M  \hookrightarrow \RR^M \]
are natural in $M$, where the functors 
\[(-)^\RR:\Cat{Man}_k \to \Cat{Top}\qquad\mbox{ and }\qquad \RR^{(-)}:\Cat{Man}^{op}_k \to \Cat{Top}
\] are obtained by restricting the exponentiation from the fibres of $\Cat{Top}$. 
\end{prop}
The proof of the preceding proposition boils down to unfolding the
arrow part of the restriction $\RR^{(-)} : \Cat{Man}^{op}_{k} \to \Cat{Top}$
of the exponentiation functors in the toposes $\Cat{Sh}(M)$, which is
identical to~\eqref{def:path} and~\eqref{def:test}. The next proposition,
however, requires that we define the covariant version
$\RR_\#^{(-)} : \Cat{Man}_{k} \to \Cat{Top}$, which is just the
exponentiation on the objects, but maps every differentiable
morphism $h\in \Cat{Man}_{k}(M,N)$ to the internal left adjoint of
$\RR^h$, i.e.
\bear
\RRD^h \dashv \RR^h &:& \RR^N\to \RR^M 
\eear

\begin{defn}
The covariant functor $\RR_\#^{(-)} : \Cat{Man}_{k} \to \Cat{Top}$ maps
each manifold $M$ to the sheaf $\RR^M$, and maps each
$h\in \Cat{Man}_{k}(M,N)$ to
\bea\label{eq-covarR}
\RRD^h(U) \ :\ \RR^M(U) & \to & \RR^N(U)\\
s & \mapsto & \RRD^h(U)_s(u,y) = 
    \bigwedge_{\varepsilon\searrow 0} \bigvee_{\Vert h(x)-y\Vert \leq \varepsilon} s(u,x) \notag 
\eea
\end{defn}

\begin{lem}\label{lem-above}
The direct images along the morphisms with nonzero differentials are
injective, modulo the approximation relation. More precisely, for every
$h\in \Cc^1(M,N)$ and every $c\in M$ such that $\bfDk h(c) \neq 0$ holds
\bea
\RR\left(\RRD(h)_t\right) & \underset c \sim& t \label{eq-R}\\
\TS\left(\TCS(h)_t\right) & \underset c \sim& t 
\eea
where we write $\RRD(h)$ for $\RRD^h$ and $\RR(g)$ for $\RR^g$.
\end{lem}

\begin{proof}[Proof of~\ref{lem-above}~\eqref{eq-R}] Omitting the
sheaf variation along $U\subseteq M$, as it plays no role here, for
all $z$ in the neighborhood of $c$ holds
\bear
\RR\left( \RRD(h)_t \right) (z) & = & \bigwedge_{\varepsilon\searrow 0}\ \  \bigvee_{\Vert h(x)-h(z)\Vert \leq \varepsilon} t(x)\ \ \underset z \sim\\
& \underset z \sim & \bigwedge_{\varepsilon\searrow 0}\ \  \bigvee_{\bfDk h(z)\Vert x-z\Vert \leq \varepsilon} t(x)\\ & = & t(z) 
\eear
\end{proof}

\begin{prop} \label{prop-covariang}
The inclusions
\[ \F M  \hookrightarrow \RRD^M \]
are natural in $M$ with respect to all \/ \emph{continuously}
differentiable morphisms $h:M\to N$.
\end{prop}

\begin{proof}
The claim is that the square
\begin{equation}\label{eq-sqr}
\begin{tikzcd}[ampersand replacement=\&]
		\F M  \arrow{r}{\modeledby} 
		\arrow[swap]{d}{\F h} 
		\& \RR^{\G M}  \arrow{d}{\RRD^f}\\
\F N \arrow[swap]{r}{\modeledby}	
\&\RR^{\G N}
\end{tikzcd}
\end{equation}
commutes for every continuously differentiable $h$. Going down and
right gives
\bea\label{eq-one}
\left(\F h (t) \underset{N} \modeledby \tau \right) & = & \bfDk\left(\bigwedge_{\varepsilon\searrow 0} \bigvee_{\Vert h(x)-\tau(r)\Vert \leq \varepsilon} t(x)  \right)_r(0)
\eea
where $t\in \F M(U)$, $\tau \in \RR^{\G N}(U)$ and the function under
the differential operator varies over $r\in \RR$. We omit the variable
$u\in U$ on both sides, since the equation holds pointwise in it. Going
right and down in~\eqref{eq-sqr} gives
\bea\label{eq-two}
\RRD^h(U)_{\widehat t} (\tau) & = & 
  \bigwedge_{\varepsilon\searrow 0} \bigvee_{\Vert h\left(\varsigma(0)\right)-\tau(0)\Vert \leq \varepsilon} \widehat{t}(\varsigma)
\eea
where the supremum is taken over $\varsigma \in \G M$ and $\widehat t \in \RR^{\G M}$
is defined to be
\[
  \widehat t(\varsigma)\ \ =\ \  \left(t\underset M \modeledby \varsigma\right)\ \ =\ \ \bfDk(t\circ \varsigma)(0) 
\] 
Substituting $\varsigma$ into $\widehat t$ and putting~\eqref{eq-one}
and~\eqref{eq-two} together, the commutativity of~\eqref{eq-sqr} thus
boils down to
\bear
  \bfDk\left(\bigwedge_{\varepsilon\searrow 0} \bigvee_{\Vert h(x)-\tau(0)\Vert \leq \varepsilon} t(x)  \right)(0) & = & \bigwedge_{\varepsilon\searrow 0} \bigvee_{\Vert h\left(\varsigma(0)\right)-\tau(0)\Vert \leq \varepsilon} \bfDk(t\circ \varsigma)(0)
\eear
the first supremum is taken over $x\in M$ and the second one over
$\varsigma \in \G M$. This equation holds as soon as $h$, $t$ and
$\varsigma$ are continuously differentiable. 
\end{proof}

\paragraph{Remark.} Note that $\bigwedge_{\varepsilon\searrow 0} \bigvee_{\Vert h(x)-y \Vert \leq \varepsilon} g(x)$
is here just $\lim\sup_{y\rightarrow h(x)} g(x)$ and that the differential
of a real function $y:\RR\to \RR$ depending on $x$ is just the slope of
its tangent, i.e. the usual derivative $\bfDk y= \frac{dy}{dx}$.

%=======================================================================
\subsection{Tangent testing correlation}\label{subsec:ttcorelation}
%=======================================================================

We are now in the position to define a testing correlation like
in~\eqref{correlation} between the paths and the tests in a manifold $M$.
\begin{defn}\label{def:testing}
The \emph{testing correlation}\/ for differentiable manifold $M$ is the
morphism $\TT$ in $\Cat{Sh}(M)$
\bea\label{eq-def-correlation}
\Sigma M \times \Theta M & \tto \TT & \RR\\
<\RR\tto \varsigma M, M\tto t \RR > & \mapsto & \bfDk(t\circ \varsigma)(0)\notag
\eea
\end{defn}

\begin{lem}\label{lem-dinat}
The testing correlations are dinatural, in the sense that every
differentiable manifold morphism $f:M\to N$ makes the following
diagram commute.
\[
\begin{tikzcd}[ampersand replacement=\&]
		\G M \times \FS N 
		\arrow{rr}{\G f\times {\rm id}} 
		\arrow[swap]{d}{{\rm id}\times \FS f} 
		\&\& \G N \times \FS N \arrow{d}{\TT}\\
\G M \times \FS M 
\arrow[swap]{rr}{\TT}	
\&\&	 \RR
\end{tikzcd}
\]
\end{lem}

\begin{proof}
$\TT\left(\varsigma, \FS f(t)\right) = \bfDk(t\circ f \circ \varsigma)(0) 
 = \TT\left(\G f(\varsigma), t\right)$
\end{proof}

Recall from Sec.~\ref{subsec:testingRel} that $\TT(\varsigma, t)$ is
also written in the infix forms 
\begin{align*}
	  &\raisebox{0.5ex}{\begin{tikzcd}[ampersand replacement=\&]
		\Sigma M \arrow{r}{\models} \& \RR^{\Theta M} 
	 \end{tikzcd}} : \quad
	   \varsigma \mapsto \left(t \mapsto (\varsigma\models t)\right)
     \\
     &\raisebox{0.5ex}{\begin{tikzcd}[ampersand replacement=\&]
		\Theta M \arrow{r}{\modeledby} \& \RR^{\Sigma M} 
	 \end{tikzcd}} : \quad
	   t \mapsto \left(\varsigma \mapsto (t \modeledby \varsigma)\right)
\end{align*}
Lemma~\ref{lem-dinat} now immediately imples

\begin{cor}\label{cor-natural}
The testing correlations induce the natural transformations
\[ \G \tto{\ \models\ } \RR^{\FS} \qquad \mbox{and} \qquad \FS \tto{\ \modeledby\ } \RR^{\G}
\]
\end{cor}
Proceeding as in~\eqref{factor}, we  define the 'observable behaviors'
of the paths in $\Sigma M$ and of the tests in $\Theta M$ by the
epi-mono factorizations in $\Cat{Top}$
\begin{align}\label{eq:GFepimono}
	\begin{tikzcd}[ampersand replacement=\&,column sep=small]
		\Sigma M \arrow{rr}{\models}
		       \arrow[two heads,swap]{rd}{e} 
		\& \& \RR^{\Theta M} \\
		\& 
		\T M \arrow[tail,swap]{ur}{m} 
		\&
	\end{tikzcd}
	&&
	 \begin{tikzcd}[ampersand replacement=\&,column sep=small]
		\Theta M \arrow{rr}{\modeledby}
		       \arrow[two heads,swap]{rd}{\overline{e}} 
		\& \& \RR^{\Sigma M} \\
		\& 
		\TS M \arrow[tail,swap]{ur}{\overline{m}} 
		\&
	\end{tikzcd}
\end{align}

\begin{prop}
If $M$ is a $k+1$ times differentiable manifold, then the sheaves
$\T M$ and $\TS M$ defined in~\eqref{eq:GFepimono} are $k$ times
differentiable manifolds: $\T M$ is just the tangent bundle of $M$,
whereas $\TS M$ is the cotangent bundle.
\end{prop}

\begin{proof}
The claim that the quotient $\T M$ defined in~\eqref{eq:GFepimono} on
the left is the same as the tangent bundle $\T M$ defined in
Sec.~\ref{sec:manifolds} formally means that for any two differentiable
curves $\varsigma, \varsigma ' \in \Cc^{k+1}(\RR,M)$ through the
same point $\varsigma(0) = \varsigma'(0)$ holds
\bear
\varsigma \underset 0 \sim \varsigma' & \iff & \forall t\in \Cc^{k+1}(M,\RR).\ \bfDk(t\circ \varsigma)(0) = \bfDk(t\circ \varsigma')(0)
\eear
But this follows from Lemma~\ref{lem-equiv}.

The claim that the quotient $\TS M$ defined in~\eqref{eq:GFepimono} on
the right is the usual cotangent bundle manifold $\TS M$, i.e. that it
consists of the cotangent vectors, dual to the tangent vectors in $\T M$,
is clear from the construction of the quotient $\TS M$, which is dual
to the construction of $\T M$.
\end{proof}

\subsection{Main theorem}
\begin{thm}\label{thm-main}
The factorizations in \eqref{eq:GFepimono} determine the functors
 \[
 \T, \TCS \ :\ \Cat{Man}_\infty \to \Cat{Man}_\infty  \qquad \mbox{and} \qquad  
 \TS\ : \  \Cat{Man}_\infty^{op} \to \Cat{Man}_\infty    \]
with $\TCS M = \TS M$ on the objects, and such that
\begin{itemize}
\item $\T$ is a monad,
\item $\TCS$ is a comonad, and
\item $\TS$ is self-adjoint. 
\end{itemize}

\end{thm}

The proof of the theorem occupies the rest of this section.

\subsubsection{The arrow part of $\T$ and $\TS$}
To define the arrow part of the functors $\T$ and $\TS$, let $M$ and
$N$ be smooth manifolds, and $f\in \Cat{Man}_\infty (M,N)$ a smooth map
between them. Consider the following diagrams:
\begin{align}\label{eq:prisms}
	&\begin{tikzcd}[ampersand replacement=\&,column sep=small]
		\G M \arrow{rr}{\models}
		   \arrow[two heads,swap]{dr}%{e_{M}}
		   \arrow[swap]{dd}{\G f}
		   \&      \& 
		   \RR^{\FS M}
		     \arrow{dd}{\RR^{\FS f}} \\
		   \& \T M 
		      \arrow[tail,swap]{ru}%{m_{M}}
		      \& \\
		\G N \arrow[pos=0.75]{rr}{\models}
		   \arrow[two heads,swap]{dr}%{e_{N}}
		   \&      \& \RR^{\FS N} \\
		   \& \T N
		      \arrow[pos=0.75,crossing over,leftarrow,dashed]{uu}{\T f} 
		      \arrow[tail,swap]{ru}%{m_{N}} 
		      \&    
	\end{tikzcd}
	&&&
	\begin{tikzcd}[ampersand replacement=\&,column sep=small]
		\FS M \arrow{rr}{\modeledby}
		   \arrow[two heads,swap]{dr}%{\overline{e}_{M}}
		   \&      \& 
		   \RR^{\G M}
      \\
		   \& \TS M 
		      \arrow[tail,swap]{ru}%{\overline{m}_{M}}
		      \&
		\\
		\FS N \arrow[pos=0.75]{rr}{\modeledby}
		   \arrow[two heads,swap]{dr}%{\overline{e}_{N}}
		   \arrow{uu}{\FS f}
		   \&      \& \RR^{\G N} 
		    		     \arrow[swap]{uu}{\RR^{\G f}}  
		\\
		   \& \TS N
		      \arrow[pos=0.75,crossing over,dashed]{uu}{\TS f} 
		      \arrow[tail,swap]{ru}%{\overline{m}_{N}} 
		      \&    
	\end{tikzcd}
\end{align}
The back squares in both diagrams commute by Corrolary~\ref{cor-natural}.
The orthogonality of the monics and the epics in the factorizatons
in~\eqref{eq:GFepimono} then induces the unique maps $\T f$ and $\TS f$.
The functoriality of $\G$ and $\FS$ guarantees the functoriality of
$\T$ and $\TS$. 

\subsubsection{The arrow part of $\TCS$}
To define the arrow part of the covariant cotangent bundle
functor $\TCS$, consider
\begin{equation}\label{eq:prism}
	\begin{tikzcd}[ampersand replacement=\&,column sep=small]
		\F M \arrow{rr}{\modeledby}
		   \arrow[two heads,swap]{dr}%{\overline{e}_{M}}
		   \&      \& 
		   \RRD^{\G M}
      \\
		   \& \TCS M 
		      \arrow[tail,swap]{ru}%{\overline{m}_{M}}
		      \&
		\\
		\F N \arrow[pos=0.75]{rr}{\modeledby}
		   \arrow[two heads,swap]{dr}%{\overline{e}_{N}}
		   \arrow[leftarrow]{uu}{\F f}
		   \&      \& \RRD^{\G N} 
		    		     \arrow[swap,leftarrow]{uu}{\RRD^{\G f}}  
		\\
		   \& \TCS N
		      \arrow[pos=0.75,crossing over,dashed,leftarrow]{uu}{\TCS f} 
		      \arrow[tail,swap]{ru}%{\overline{m}_{N}} 
		      \&    
	\end{tikzcd}
\end{equation}
This time, the back square commutes by Prop.~\ref{prop-covariang}, and
the functor thus depends on the direct images in the definitons of
$\F$ and $\RRD^{\G(-)}$. The uniqueness and the functoriality of the
definition of $\TCS f$ follow from the same reasons as above.

\subsubsection{The bundle map \ $\T \tto\varepsilon {\rm Id}$}\label{lem:Tepsilon}
%\begin{lem}
The bundle projection $\varepsilon: \T M\to M$ is constructed using the
orthogonality of the monics and the epis again
\begin{align*}
	\begin{tikzcd}[ampersand replacement=\&]
		\G M \arrow[two heads]{r}{e} 
		   \arrow[swap]{dr}{\epsilon} 
		\& \T M \arrow[tail]{r}{m}
		        \arrow[dashed]{d}{\varepsilon}
		\& \RR^{\FS M} \arrow{d}{\RR^{\modeledby}_{!}}
		  \\
	  \&
		M \arrow[tail]{r}{m'}
		  \&
		\RR^{\RR^{\G M}}         
	\end{tikzcd}
\end{align*}
where the top is the definition of $\T M$, the diagonal map is the projection
\bea\label{eq-epsilon}
\epsilon(\varsigma) &= & \varsigma(0)
\eea
and the incusion $m'$ is defined as follows:
\begin{align*}
   \inferrule*{%
     \begin{tikzcd}[ampersand replacement=\&, column sep=1.5cm]
		 \RR^{\G M}\times M \arrow{r}{\RR^{\displaystyle\iota}\times M} \&
		 \RR^{M} \times M \arrow{r} \&
		 \RR
	  \end{tikzcd}	}{
	  \begin{tikzcd}[ampersand replacement=\&]
        M \arrow{r}{m'} \& \RR^{\RR^{\G M}}		     
	  \end{tikzcd}}
\end{align*}
%\end{lem}

The map $\RRD^{\modeledby}$ is the image of the testing correlation
along the covariant exponentiation functor from~\eqref{eq-covarR},
i.e.
\begin{align}\label{eq:Tcoevalualtion}
	\inferrule*{%
	  \begin{tikzcd}[ampersand replacement=\&]
		  \FS M \arrow{r}{\modeledby} \& \RR^{\G M}
	  \end{tikzcd}}{%
	  \begin{tikzcd}[ampersand replacement=\&]
		  \RR^{\FS M} \arrow[<-,transform canvas={yshift=.5ex}]{r}{\RR^{\modeledby}}
		           \arrow[swap,transform canvas={yshift=-.5ex}]{r}{\RR^{\modeledby}_{!}} \& \RR^{\RR^{\G M}}
	  \end{tikzcd}}
\end{align}

The natural transformation $\delta: \T M\to \T \T M$ can be derived using the orthogonality of the factorization: 
\begin{align*}
	\begin{tikzcd}[ampersand replacement=\&]
		\G M 
		\arrow[two heads]{rr}{e_M} 
		   \arrow[swap]{dd}{\gamma} \&\& 
		   \T M \arrow[tail]{rr}{m_M} \arrow[dashed]{dd}{\delta} \&\& 
		   \RR^{\FS M} \arrow{d}{\RRD^{\displaystyle\FS \epsilon}}\\
		   \& \&
		\& \&
		\RR^{\FS \G M} \arrow{d}{\RR^{\displaystyle\FS e_M}}
		  \\
\G \T M\arrow[two heads,swap]{rr}{e_{\T M}}	  \&\&
		\T \T M \arrow[tail,swap]{rr}{m_{\T M}}
		  \&\&
		\RR^{\FS \T M}         
	\end{tikzcd}
\end{align*}
where
\bear
\gamma(\varsigma)(y) &= & \varsigma  +\bfDk \varsigma(0) \cdot y
\eear
and $\epsilon$ is as in~\eqref{eq-epsilon}. This is not a comonad coevaluation, as it does not split $\varepsilon \T$.

\subsubsection{The monad structure \ $\T\T\tto \mu \T \oot \eta {\rm Id}$}
%\begin{lem}
The monad structure on $\T$ is induced by 
\begin{itemize}
\item the bundle projections $\varepsilon: \T M\to M$ from Sec~\ref{lem:Tepsilon} and the
\item the 'candidate units' $\eta_\Sigma : M \tto \Sigma M$ where $\eta_\Sigma x (r) = x$
\end{itemize}
%The unit and evaluation
%structure on $\TCS$ is induced by the manifold metrics. 
\bear
\begin{tikzcd}[ampersand replacement=\&,column sep=small]
\& M \arrow[dashed]{d}{\eta_{\T}} \arrow[swap]{dl}{\eta_\Sigma} \\
\Sigma M \arrow[two heads]{r} \& \T M \arrow[tail]{r} \& \RR^{\Theta M} \\
\Sigma \T M \arrow{u}{\Sigma \varepsilon} \arrow[two heads]{r} \& \T \T M \arrow[dashed,swap]{u}{\mu_{\T}} \arrow[tail]{r} \& \RR^{\Theta \T M} \arrow[swap]{u}{\RR^{\FS \varepsilon}}
	 \end{tikzcd}
%	&& 
%	 \begin{tikzcd}[ampersand replacement=\&,column sep=small]
%\& M \arrow[dashed]{d}{\eta_{\TCS}} \arrow[swap]{dl}{\eta_\Theta} \\
%\Theta M \arrow[two heads]{r} \& \TCS M \arrow[tail]{r} \& \RR^{\Sigma M} \\
%\Theta \TCS M \arrow{u}{\F \varepsilon} \arrow[two heads]{r} \& \TCS \TCS M \arrow[dashed,swap]{u}{\mu_{\TCS}} \arrow[tail]{r} \& \RR^{\Sigma\TCS M} \arrow[swap]{u}{\RRD^{\G \varepsilon}}
%	 \end{tikzcd}
	 \eear

%\end{lem}
%
%\begin{proof}[Proof sketch]
%The structures are induced on the following diagram
%\bear
%\begin{tikzcd}[ampersand replacement=\&,column sep=small]
%\& M \arrow[dashed]{d}{\eta_{\T}} \arrow[swap]{dl}{\eta_\Sigma} \\
%\Sigma M \arrow[two heads]{r} \& \T M \arrow[tail]{r} \& \RR^{\Theta M} \\
%\Sigma \T M \arrow{u}{\Sigma \varepsilon} \arrow[two heads]{r} \& \T \T M \arrow[dashed,swap]{u}{\mu_{\T}} \arrow[tail]{r} \& \RR^{\Theta \T M} \arrow[swap]{u}{\RR^{\FS \varepsilon}}
%	 \end{tikzcd}
%%	&& 
%%	 \begin{tikzcd}[ampersand replacement=\&,column sep=small]
%%\& M \arrow[dashed]{d}{\eta_{\TCS}} \arrow[swap]{dl}{\eta_\Theta} \\
%%\Theta M \arrow[two heads]{r} \& \TCS M \arrow[tail]{r} \& \RR^{\Sigma M} \\
%%\Theta \TCS M \arrow{u}{\F \varepsilon} \arrow[two heads]{r} \& \TCS \TCS M \arrow[dashed,swap]{u}{\mu_{\TCS}} \arrow[tail]{r} \& \RR^{\Sigma\TCS M} \arrow[swap]{u}{\RRD^{\G \varepsilon}}
%%	 \end{tikzcd}
%	 \eear
%where
%$
%	  \raisebox{0.5ex}{
%	  \begin{tikzcd}[ampersand replacement=\&]
%		M \arrow{r}{\eta_\Sigma} \& \Sigma M \subseteq M^\RR 
%	 \end{tikzcd}
%	 } \ \  : \ \ 
%	   x \mapsto \eta_\Sigma x (r) = x
%%     \\
%%     \raisebox{0.5ex}{
%%     \begin{tikzcd}[ampersand replacement=\&]
%%		M \arrow{r}{\eta_\Theta} \& \Theta M\subseteq \RR^{M} 
%%	 \end{tikzcd}
%%	 } &: \quad
%%	   x \mapsto \eta_\Theta x(y) = \Vert x-y\Vert
%$
%\end{proof}

\paragraph{Remark.} The monad structure over $\T$ has been spelled out in the
classical framework in \cite{jubin:2012a}, and studied in
of Synthetic Differential Geometry in~\cite{Cockett:Cruttwell:2012a}.

\subsubsection{The comonad structure \ $\TCS\TCS \oot \delta \TCS \tto \varepsilon {\rm \Id}$}
Towards the definition of the bundle structure map
$\varepsilon : \TCS M\to M$, define for each $x\in M$ the subsheaf
$\Sigma_x M$ of $\G M$ by setting
\bea
\Sigma_x M & = & \left\{\varsigma \in \G M\ |\ \varsigma(0) = x\right\}
\eea
Since $\G M$ thus decomposes into the coproduct
\bear
\G M & = & \coprod_{x\in M} \Sigma_x M
\eear
the cotangent bundle embeds into the product
\begin{align}\label{eq:TCS-prod}
	 \begin{tikzcd}[ampersand replacement=\&,column sep=small]
		\Theta M \arrow{rr}{\modeledby}
		       \arrow[two heads]{rrd}
	 \& \& \RR^{{\coprod \Sigma_z M}} \arrow{rr}{\simeq} \& \&  {\displaystyle \prod_{z\in M} \RR^{\Sigma_z M}}  \arrow[two heads,swap]{dd}{\pi_x}\\
		\& \&
		\TS M \arrow[tail]{urr} 
		\arrow[two heads,swap]{dd}{p_x}
		\\
		\&\&\&\& \RR^{\Sigma_x M} \arrow[bend right,tail,swap]{uu}{\iota_x}
		\\
		\& \& T^x M \arrow[tail]{urr}
	\end{tikzcd}
\end{align}
The quotient bundle $T^x M$ of $\TS M = \TCS M$ thus consists of the
equivalence classes of the tests on $M$ that are indistinguishable by
the paths through $x$. Note, however, that, since $\RR$ is an
injective object of $\Cat{Sh}(M)$, the projections
$\prod_{z\in M} \RR^{\Sigma_z M} \tto{\pi_x} \RR^{\Sigma_x M}$ split
by the injections $\iota_x \ :\ \RR^{\Sigma_x M} \tto{\iota_x} \prod_{z\in M} \RR^{\Sigma_z M}$.
The splitting means that $\pi_x\circ \iota_x = {\rm id}$ holds.
Intuitively, $\iota_x$ maps $\tau\in \RR^{\Sigma_x M}$ to the
family $\iota_x (\tau) = \left<\iota \tau_z\right>_{z\in M}$ where
$\iota \tau_z = \tau$ if $z=x$ and $\iota \tau_z = 0$ otherwise.
$T^x M$ is thus a submanifold of $\prod_{z\in M} \RR^{\Sigma_z M}$,
and hence
\begin{equation}
\begin{tikzcd}[ampersand replacement=\&]
  \Theta M
  \arrow{dr}
  \arrow[two heads]{d}  \\
 T^x M \arrow[tail]{r} \arrow[dashed,tail]{d} 
\& \RR^{\Sigma_x M} \arrow[tail]{d}{\iota_x} \\
\TS M \arrow[tail]{r} 
\&
{\displaystyle \prod_{z\in M} \RR^{\Sigma_z M}}
\end{tikzcd}
\end{equation}
Since the inclusions $T^x M\hookrightarrow \TS M$ are jointly
surjective, we have
\bear
\TS M & = & \coprod_{x\in M} T^x M
\eear
Recalling that $\TCS$ is the covariant version of $\TS$, i.e. that
$\TCS M=\TS M$ on the objects, we construct $\varepsilon : \TCS M\to M$
by projecting each $T^x M$ to $x\in M$. The coevaluation
$\delta:\TCS M\to \TCS\TCS M$ follows by using the contravariant
functors $\FS$ and $\RR^{\G}$.
\begin{equation}
\begin{tikzcd}[ampersand replacement=\&]
  \FS M
  \arrow[swap]{d}{\FS \varepsilon}
  \arrow[two heads]{r}  
  \&
 \TCS M \arrow[dashed,swap]{d}{\delta} \arrow[tail]{r} 
\& 
\RR^{\G M} \arrow{d}{\RR^{\G \varepsilon}} \\
\FS \TCS M \arrow[two heads]{r} 
\&
\TCS \TCS M \arrow[tail]{r} \& \RR^{\G\TCS M}
\end{tikzcd}
\end{equation}
The naturality and the comonad conditions for the families
\[
  M \oot \varepsilon \TCS M \tto \delta \TCS \TCS M
\]
follow from Lemma~\ref{lem-above}.

\paragraph{Remark.} Just like the tangent bundle functor allowed a 'coevaluation candidate' $\T \tto \delta \T\T$ which turned out to be degenerate, the cotangent bundle allows a degenerate 'evaluation candidate' $\TCS \TCS \tto \mu \TCS$. While both candidacies fail, they may be pointing towards a structure that still waits to be understood.

\subsubsection{$\TS$ is self-adjoint}
The claim is that there is a correspondence
\bear
\Cat{Man}_\infty(M, \TS N) & \cong & \Cat{Man}_\infty\left(N,\TS M\right)
\eear
natural in $M$ and $N$. By the definition in~\eqref{eq:GFepimono}, the
manifold $\TS N$ consists of the equivalence classes of tests in
$\RR^N$ indistinuishable by smooth paths in $N^\RR$. On the other hand,
two smooth morphisms $f,g\in \Cat{Man}_\infty(M, M')$ that are
indistinguishable by smooth paths in $M'$ must be equal. So a morphism
$f\in \Cat{Man}_\infty(M, \TS N)$, unfolded to a function
$\widehat f : M\times N\to \RR$, is a representative of the equivalence
classes modulo the same testing along each of the arguments. Although
this function is not a manifold morphism (and the category $\Cat{Man}_\infty${}
is not cartesian closed), it is also a representative of the corresponding
manifold morphism $f'\in \Cat{Man}_\infty(N,\TS M)$.

%=======================================================================
\subsection{Smooth semantic connection}\label{subsec:testdyn}

The structures described so far form a semantic connection in the sense
of Sec.~\ref{sec:generaltesting} as follows
\begin{itemize}
\item the category of smooth manifolds and smooth morphisms provides
      both the universe of 'spaces' to define systems, and the universe
      for 'types' to carry tests, i.e.
\[
  \SSS\ =\ \TTT\ =\  \Cat{Man}_\infty
\]
\item the 'model/predicate' adjuntion $M^{op}\dashv P : \Cal{S}^{op}\rightarrow \Cal{T}$
      is given by the cotangent bundle construction
\[
  (\TS)^{op}\dashv \TS : \Cat{Man}_\infty^{op}\rightarrow \Cat{Man}_\infty
\]
\item the `tests functor' $\Theta$ is the cotangent bundle functor $\TCS : \Man_\infty \to \Man_\infty$,
\item the `process functor' $\Sigma$ is the tangent bundle functor $\T :\Man_\infty \to \Man_\infty$,
\item the distributivity law $\lambda \,:\, \Theta P\to P\Sigma$ is derived as follows
\begin{equation}\label{eq:DitsLawDistance}
\prooftree
\prooftree
\RR^{\TS} \times \T \tto{{\rm id}\times \models} \RR^{\TS}\times \RR^{\TS} \tto{d} \RR
\justifies
\lambda' \,:\, \T \to \RR^{\RR^{T^*}} \epi \TS \TCS \TS
\endprooftree
\justifies
\lambda \,:\, \TCS \TS \to \TS \T
\endprooftree
\end{equation}
where $d(x,y) = \Vert x-y\Vert$ is the distance on the manifold, and
$\models$ is the testing correlation between the tangents and the
cotangents, induced as follows
\begin{equation}\label{eq-tantest}
	\begin{tikzcd}[ampersand replacement=\&]
		\Sigma M \arrow[bend left]{rrrr}{\models}
		\arrow[two heads]{rr}{e} 
		   \arrow[swap]{dd}{\eta} \&\& 
		   \T M \arrow[tail]{rr}{m} \arrow[tail,dashed]{dd}{\models} \&\& 
		   \RR^{\Theta M} \ar[leftrightarrow]{dd}[description]{=}\\
		  \\
\RR^{\RR^{\Sigma M}}\arrow[swap]{rr}{\RR^{\overline e}} \arrow[bend right,swap]{rrrr}{\RR^{\modeledby}}	  \&\&
		\RR^{\TS M} \arrow[tail,swap]{rr}{\RR^{\overline m}}
		  \&\&
		\RR^{\Theta M}         
	\end{tikzcd}
\end{equation}
\end{itemize}

%\include{TTcotangents}
%=======================================================================
% Obtaining classical tangent and cotangent structures
%-----------------------------------------------------------------------
% !TEX root = manifolds.tex
%=======================================================================
\section{Examples and applications}\label{sec:Newton}
\subsection{Newtonian physics using coalgebras}

We work in $\Cat{Man}_\infty$ with smooth maps. A \emph{dynamical system} modeled on $M$ is a coalgebra $X : M \rightarrow \T M$. 
\begin{defn}
	The model category $\Cat{Dyn}$ of \emph{dynamical systems} has as
	objects dynamical systems, that is smooth sections $X : M \rightarrow \T M$
	of the tangent bundle, and as morphisms $\la f : M \rightarrow N, \T f\ra$
	such that
   \begin{align*}
		\begin{tikzcd}[ampersand replacement=\&]
		M \arrow{r}{f}
		  \arrow{d}{X}
		  \& 
		N \arrow{d}{Y} 
		\\
		\T M \arrow{r}{\T f}
		   \& 
		\T N
	\end{tikzcd}   
	\end{align*}
	commutes. The dynamical systems $X$ and $Y$ are called $f$-related.
	The identity and composition are induced from $\Cat{Man}_\infty$.	
\end{defn}
$\Cat{Dyn}$ is the category of $\T$-coalgebras $\Cat{Man}_{\T}$. 
The path category $\G\Cat{Man}$ is a full subcategory of $\Cat{Dyn}$
where the objects are in the form $R : I \rightarrow \T I$ for open intervals $I\subset\RR$.
We have $R(t)=(t,1)$ that is $\dot x=\textrm{d}x/\textrm{d}t=1$ represents
a clock\footnote{%
The clock vector field is denoted $R$ because the Greek word for clock is $\rho o\lambda \grave o\ddot\iota$.}
at unit rate, setting a unit time interval. $\T I\cong I\times \RR$ is
trivializable. We will assume that $I$ contains the point $0$.
\begin{defn}
	A trajectory of a dynamical system is a morphism $\gamma : R \rightarrow X$
	from the `clock` morphism $R : I \rightarrow TI$ to the dynamical system
	$X : M \rightarrow \T M$ such that
	\begin{align*}
		\begin{tikzcd}[ampersand replacement=\&]
		I \arrow{r}{\gamma}
		  \arrow{d}{R}
		  \& 
		M \arrow{d}{X} 
		\\
		\T I \arrow{r}{\T\gamma}
		   \& 
		\T M
	\end{tikzcd}   
	\end{align*}
	commutes.
\end{defn}
\paragraph{Remark.}
	$R$ describes the eigentime of a test particle (its internal clock)
	moving on the path. A suitable (monotone) map $R\mapsto R'$ would
	describe the movement of the test particle with a different clockrate,
	hence points towards special relativity. We stick to classical
	mechanics here.

From $X\circ\gamma = \T\gamma\circ R$ one obtains $(\gamma(t),\bfDk\gamma(t)\cdot 1)
=(\gamma(t), X(\gamma(t))\cdot 1)$ and hence a first order differential
equation $\dot\gamma(t)=X(\gamma(t))$ for $\gamma(t)$ with initial condition
$q=\gamma(0)$. Such a curve is an \emph{integral curve}. It can be
integrated for some parameter ranges and initial conditions. The integral curve
is unique on a maximal parameter range $I=(-a,a)$~\cite{Kolar:Michor:Slovak:1993a}.

From coalgebras $x : X \rightarrow \T X$ we get first order ordinary
differential equations of the prototypical form $\dot{x}(t)=f(x(t))$. A
formal solution is given by direct integration $x(t)=\int f(x(t)) = \Phi(x(t))$.
The integral operator is (circumstances permitting) a Fredholm operator
and allows an iterative solution (resolvents and Neumann series).
More prominently a solution to the ordinary differential equation is
a fix point for the integral operator $\Phi$, hence integral curves for
a vector field (flow) are \emph{fix points}.

\paragraph{Newtonian gravity.}
Our running example will be Newton's law of gravity. We use the following
notations: $m$ is the mass (constant non relativistic) of a body, $r$ its
position with base point $0$ in a Euclidean space 
$E=\RR^{3}\setminus \{0\}$, $\hat{r}_{12}=r_{12}/\vert r_{12}\vert$ is the
unit vector in direction of $r_{12}$, $v=\dot r = \textrm{d}r/\textrm{d}t$
is the velocity of a body, $p = \textrm{d}(m\,v)/\textrm{d}t =m\dot v$ the
(linear) momentum, and $a=\dot v$ the acceleration. We put a source (mass)
$m_{1}$ at the origin. The gravitational potential 
$\phi(r_{12})=G\frac{m_{1}}{\vert r_{12}\vert}$ (potential energy per
unit mass, zero at $\vert r_{12}\vert=\infty$) gives rise to the
gravitational force\footnote{%
Technically, by Newton's third law (see below), the source $m_{1}$ is
moved by the reaction force of the test mass $m_{2}$, and one uses the
center of mass and the reduced two body problem. 
If $m_{1}\gg m_{2}$ one can ignore this.}
$F(r_{12})=\la -\nabla m_{2}\phi(r_{12}),\hat{r}_{12}\ra \hat{r}_{12}
= -m_{2}\bfDk\phi(r_{12})$ on a test body $m_{2}$ at position $r_{12}$.
Here we use the relation between gradient and (directional) derivative
$(\nabla f(x))\cdot v = (\bfDk_{v})f(x) = \bfDk f(x)\cdot v$.
$F\in \Gamma \T X$ is a section of the tangent bundle and assigns to each
$r_{12}$ a linear curve ($i_{UX}$ is the inverse of a trivial chart)
\begin{align}
	&\begin{tikzcd}[ampersand replacement=\&]
		E \arrow{r}{F}
		 \& 
		\T X 
		\\
		U \arrow[tail]{u}{i_{UX}}
		  \arrow{r}{F_{U}}
		 \& 
		U\times \RR^{3}
		  \arrow[tail,swap]{u}{\la i_{UX}\circ\pi_{0},\bfDk i_{UX} \ra}
	\end{tikzcd}
	\\
	F_{U}&=\big\la r_{12}, -G\frac{m_{1}m_{2}}{\vert r_{12}\vert^{3}} r_{12}\big\ra
	      =\big\la r_{12}, -m_{2}\bfDk \phi(r_{12})\big\ra     
\end{align}

\paragraph{Newton's second law.}
With the notation from above, Newtonian physics is based on three postulates:
\begin{description}
	\item[\textbf{NP1:}] For zero external force, any body moves on a straight
	   line with constant velocity. That is $F=0 \Rightarrow \dot{v}=0$.
	\item[\textbf{NP2:}] The change of momentum is equivalent to the external
	   force acting on the body. The law of motion is $\dot{p}=F$ with
	   $p=mv$ and constant mass ($\dot{p}=m\dot{v})$.
	\item[\textbf{NP3:}] For any two bodies $m_{i}$, $1\leq i\leq 2$, their
	   interaction is such that the force $F_{12}$ (1 acts on 2) and $F_{21}$
	   (2 acts on 1) are summing to zero $F_{12}+F_{21}=0$. Action
	   equals reaction.   
\end{description}
Newton's second law implies
\begin{align}
	\label{eq:Lagrangian}
	\frac{\textrm{d}r}{\textrm{d}t}& = v
	&&&
	\frac{\textrm{d}v}{\textrm{d}t}& = \frac{1}{m}F(r)
	&&&&\textrm{Lagrangian or configuration space view}
	   \\
		\label{eq:Hamiltonian}
	\frac{\textrm{d}r}{\textrm{d}t}& = \frac{p}{m}
	&&&
	\frac{\textrm{d}p}{\textrm{d}t}& = F(r)
	&&&&\textrm{Hamiltonian or phase space view.}
\end{align}
The difference is that we look at different bundle structures. One can
switch between these views taking into account the `evaluation' (actually
using the Hamiltonian and the symplectic form)
\begin{align*}
  \TCS M \times \T M &\rightarrow \RR :: 
  (r,p)\times (r,v) \mapsto \frac{1}{2} p(v) = \frac{1}{2} \la m v,v\ra	
\end{align*}
by the Riesz representation theorem. It involves the mass of the test
body $m=m_{2}$ and describes the kinetic energy $E_{kin}$ of the test body.

For~\eqref{eq:Lagrangian} consider the tangent bundle $F = 
\RR^{3}\setminus\{0\}\times \RR^{3}=\T E$. Its points are states of
the system $(r,v)\in \T^{r}X$, and $\T X$ is called 
\emph{configuration space}. The configuration space is a manifold,
and the law of motion induces a curve in it, that is we get
(using the names for maps in a local chart)
\begin{align*}
	\begin{tikzcd}[ampersand replacement=\&, column sep=1.5cm]
		I 
		  \arrow{r}{r}
		  \arrow[swap]{d}{R=\la 1,\bfDk\ra}
		 \& 
		E
		  \arrow{r}{\la 1,v\circ r^{-1}\ra}
		  \arrow[swap]{d}{\la 1,v\circ r^{-1}\ra}
       \&
		\T E
		  \arrow{d}{\big\la \la \pi_{0},\pi_{1}\ra,\la \pi_{1},\frac{1}{m}F\circ\pi_{0} \ra\big\ra}   
		\\
		\T I 
		  \arrow[swap]{r}{\T r}
		 \&
		\T E
		  \arrow[swap]{r}{T\la 1,v\circ r^{-1}\ra}
		  \arrow[<->, dashed]{ru}
		 \&
		\T\T E
	\end{tikzcd}
\end{align*}
The commutativity of the left square $\T r\circ R = \la 1, v\circ r^{-1}\ra\circ r$
gives $(r(t),\frac{\textrm{d}r}{\textrm{d}t}(t)) = (\T r\circ R)(t) =
(\la 1, v\circ r^{-1}\ra\circ r)(t) = (r(t),v(t))$, that is the lhs
of~\eqref{eq:Lagrangian}. The commutativity of the right square
boils down to the commutativity of the lower triangle as the dashed
arrow is identity. Hence one gets $((r,v),(v,\frac{\textrm{d}v}{\textrm{d}t}))
= (\T\la 1, v \circ r^{-1}\ra)(r,v) = \big\la \la \pi_{0},\pi_{1}\ra,\la \pi_{1},
  \frac{1}{m}F\circ\pi_{0} \ra\big\ra (r,v)
= ((r,v),(v,\frac{1}{m}F(r)))$ and hence the rhs
of~\eqref{eq:Lagrangian}. A solution to this differential equation
for the flow of the gravitational field is an integral curve in the
double tangent bundle $F(r,v) = ((r,v),(v,\frac{1}{m}F(r)))$. The
canonical flip acts trivially on this vector field. 

For~\eqref{eq:Hamiltonian} consider the cotangent bundle $\TCS E = 
\RR^{3}\setminus\{0\}\times \RR^{3}$. Its points are states of
the system $(r,p)\in \TCS^{r}E$, and $\TCS^{r}E$ is called 
\emph{phase space}. The phase space is a manifold, and the law of motion
induces a curve in it, that is we get
\begin{align*}
	\begin{tikzcd}[ampersand replacement=\&, column sep=1.5cm]
		E
		  \arrow{r}{\la 1,\frac{p}{m}\circ r^{-1}\ra}
		  \arrow[swap]{d}{\la 1,\frac{p}{m}\circ r^{-1}\ra}
		 \&
		\TCS E
		  \arrow{d}{\big\la \la \pi_{0},\pi_{1}\ra,\la \pi_{1},F\circ\pi_{0} \ra\big\ra}
		\\
		\TCS E
   	  \arrow[swap]{r}{\T(\la 1,\frac{p}{m}\circ r^{-1}\ra)}
   	  \arrow[<->, dashed]{ru}
		 \&
		\T\TCS E
	\end{tikzcd}
\end{align*}
The corresponding vector field on phase space is given
in~\cite[p.240]{Abraham:Marsden:Ratiu:1988a}
and reads $F(r,p) = ((r,p),(p/m,F(r))) \in \T\TCS E$. Note that unlike
for the double tangent bundle $\T^{2}X$ we cannot expect the canonical
flip to act here trivially as a duality $\T X$ versus $\TCS X$ is
involved. 
%\paragraph{Remark.}
%Although tangents and cotangents are dual, and the bundles $\T M$ and $\TS M$ are dual, this duality is not natural, just like a finitely dimensional vector space and its dual are isomorphic, but not naturally isomorphic. In presence of a non degenerate inner product or a (pseudo) Riemannian
%	metric one can define the `musical morphisms' as the left and right
%	curried versions of the duality. In classical mechanics the symplectic isomorphism $\T M \cong \TS M$.
%	Using the Riesz representation theorem one is then able to identify
%	$\TS M \cong \RR^{\T M}$. In the non relativistic example we
%	introduced the test mass $m_{2}$ in this way. In general relativity
%	we have non trivial charts, but also a metric allowing to make the above
%	identification. All these maps have to interact coherently with such
%	structures as additive bundles, vertical lift and canonical flip,
%	that is with the whole testing framework.
%	 To exemplify these statements
%	is, however, aiming into a different direction than this paper which
%	focused on tangent testing semantics. 

\subsection{Other applications}
So far, this work focused on the correlations of tangents and cotangents over the \emph{same}\/ manifold. But the examples that arise in most applications involve two different manifolds: one to present the vector field, the other one to present the form on which it acts.
Interesting questions that open up in this direction include:

%=======================================================================
\paragraph{Testing and monoidal (multilinear) structures.}
The example of Newtonian gravity in Section~\ref{sec:Newton} describes
one test particle in an external gravitational field of a source
(mass). A natural question is to ask what happens when one looks at
many particle systems. In such a situation one needs to study paths
$\gamma_{i}(t)$ parameterized over an index set $I$. That is one looks
at products of coalgebra maps each describing a path for a single
particle in the presence of the others. Alternatively one can study
the path of a single system in a higher dimensional space. E.g. two
particles moving in $\RR^3$ can be described as single system moving
in $\RR^6$. Categorically this asks how the (co)tangent functors
interact with sums and products $\T(E + F)$ versus $\T E \otimes \T F$.
Moreover, looking at the monad evaluation $\T\T M \tto \delta \T M$ is
given by the sum $\mu=+\circ u_{M}=+\circ\la p_{T},\T p\ra$ of the two
projections from the double tangent bundle to the pullback bundle. The
canonical flip, describing the commutativity of partial derivatives,
and hence providing a local integrability condition, ensures that this
can be done. What really is needed is a monoidal structure on the
(co)tangent bundles which interacts properly with the lifting result
of the Representation Theorem~\ref{thm:representation}. This would
allow to study the categorical backbone of $k$-forms for $k\ge 1$. This
is important, as for example electromagnetism, and in general gauge
theories, are based on (curvature) 2-forms, e.g. the Maxwell tensor
$F$ is a closed 2-form.

%=======================================================================
\paragraph{Symplectic structure, general relativity.}
The testing framework yields in the differential geometry setting very
naturally the tangent $\T$, $\TCS$ and cotangent $\TS$ functors. On (finite)
linear fiber spaces there is a closed structure available. However, in
physical applications one uses also noncanonical identifications of
these spaces.

A \emph{symplectic form}\/ is an isomorphism $\omega : \T M \leftrightarrow \TS M : \Omega$.
Using the closed structure one has $\omega \in \TS M \wedge \TS M$.
This 2-form plays a fundamental role in dynamical systems, e.g.
transformations leaving this form invariant are \emph{symplectomorphisms}\/
inducing \emph{canonical transformations}\/ and do not change the nature
of a physical system. The symplectic form
provides also a description of time evolution of systems via the
Hamiltonian $H$. Let $H : M\rightarrow \RR$ be differentiable map, and
consider the 1-form $\rmd H$, which acts via the closed structure on
$Y\in \T M$. Then one can define a Hamiltonian vector field $X_{H} \in \T M$
via $\rmd H(Y) = \omega(X_{H},Y)$. The Hamiltonian vector field describes
the evolution of the system and implies the Hamiltonian form of the
equations of motion. The interaction of the symplectic form, that is
the closed structure, and the testing framework providing fixed point
operators should allow to study general aspects of dynamical systems
in a new way. 

Given a (locally Minkowski) metric $g$ on the manifold allows to put a
metric on the linear models $g : \T M \times \T M \rightarrow \RR$. Via
the Riesz representation theorem this identifies also $\T M$ and $\TS M$.
We have seen that dynamical systems arise from the lifting of the paths
functor $\Sigma$ to $\T$-coalgebras. Hence to make the metric a dynamical
object, one needs a similar lifting for the situation in the diagram of
the last paragraph with the metric identification of $\TS$ with $\T$
and wise versa. Such a dynamic involves curvature as we are now dealing
with higher forms, and it should be possible to employ the testing
framework in a general relativistic setup. Note also that the distributivity
law \eqref{eq:DitsLawDistance} depends on the distance function
$d : \RR^{\TS}\times \RR^{\TS} \rightarrow \RR$ allowing to include a metric.

%=======================================================================
\paragraph{Quantum mechanics.}
The testing framework does not choose bases, and also does not need in
principle finite dimensional (co)tangent spaces. It is hence general
enough to model infinite dimensional Hilbert spaces. Testing on such
spaces can no longer be done point wise (Heisenberg uncertainty) and
needs test functions. Using compactly supported or Schwartz test
functions destroys the closed structure on the vector bundles as duals
of such function spaces are strictly larger. Moreover, certain operators
like the momentum operator are no longer bounded (continuous). Restricting
the tests in the test category still allows to test systems, at the
cost to be able to detect only such behavior which is distinguishable
by the chosen tests. One may expect that a coalgebraic study of this
situation can reveal new insights not easy to gain by rather concrete
coordinate based methods.

%=======================================================================
\paragraph{Scalings in chaotic systems and quantum field theory.}
Another way to employ coalgebraic methods in dynamical systems comes from
studying fixed points and iterative solutions of processes. Following
iteratively a tangent vector field one obtains an approximation of
integral curves (numerical integration). Suppose $f : E \rightarrow E$
is such a map. Let $f=f^{1}$ have a fixed point $f(a)=a$, and define
$f^{n}=f\circ f^{n-1}$ the $n$-fold iteration of $f$. One
defines formally the \emph{iteration velocity}
$v(x) = \partial f^{n}(x)/\partial n\vert_{n=0}$. 
Let $g : E\rightarrow E$ be another function. In many applications one
is interested in the composition $g(f(x))$. Of course, $g(f(a))=g(a)$
for the fixed point. In nonlinear dynamical systems and in quantum field
theory at such fixed points the system enjoys a \emph{scale invariance}.
However, apart from it the system changes behavior for different scales.
Using the iteration velocity as a (tangent) vector field, one obtains a
solution for the composition problem
\begin{align*}
	g(f(x)) &= \exp\left( v(x)\frac{\partial}{\partial x}\right)g(x)
\end{align*}
The flow of this vector field is a renormalization flow of the system.
In quantum field theory $v(x)$ is the beta-function and determines the
`running' (varying) coupling strength at a different energy scales.
It seems plausible that techniques well established in computer science,
e.g. final coalgebras etc, when employed via testing, could help studying
these situations on a categorical level.

%=======================================================================
% Further ideas
%-----------------------------------------------------------------------
% !TEX root = manifolds.tex
%=======================================================================
\section{Summary and future work}\label{sec:furtherideas}
Since its inception, category theory has played a crucial role in algebraic topology and algebraic geometry. There have many categorical approaches to vector calculus and to differential geometry, synthetic \cite{KockA:SDG} and classical \cite{kriegl:frolicher:1988}. There is a lot of current work, some of which we  mentioned in the introduction. To the wealth of ideas and techniques resulting from all that work, we add a relatively minor conceptual wrinkle: semantical connections of vector fields and differential forms, and testing one over the other. On one hand, it is just a matter of presentation: we just proposed a coalgebraic reconstruction of standard concepts. This points towards a new realm of research in coalgebra, but it is not obvious that this research leads towards some genuinely new results in vector calculus or differential geometry. On the other hand, a high level overview, abstracting away the irrelevant implementation details of a complicated mathematical theory can sometimes be even more useful than a particular mathematical result. This is not often the case, but it gave us a reason to pursue the ideas of smooth coalgebra, and of testing in vector analysis.

We started from the view of vector fields as coalgebras $M\to \T M$ splitting the tangent bundle projection $\T M\tto\varepsilon M$. This is just the coalgebraic form of the familiar cross section requirement, which assures that the corresponding vector fields are locally determined. Coalgebras not satisfying this requirement permit nonlocal interactions. Differential forms similarly arise as the crossections $M\to \TCS M$ of the cotangent bundle projection $\TCS M\tto \varepsilon M$. The Main Theorem opened up the alley towards testing vector fields and differential forms. But the actual explorations of that alley are left for the future work. 

So what did we actually do here? We defined the sheaves $\Sigma M$ and $\Theta M$, consisting respectively of the paths and of the tests over a manifold $M$, and established the canonical testing correlation between them $\Sigma M \times \Theta M \tto \TT \RR$, mapping a path $\varsigma: \RR\to M$ and a test $t:M\to \RR$ to $\bfDk(t\circ \varsigma)(0)$. Applying the testing factorization \eqref{eq:epimono} to $\TT$, we then derived the tangent bundle $\T M$, as the sheaf of paths indistinguishable under the tests, and the  cotangent bundle functor $\TS M$, as the sheaf of tests indistinguishable under the paths. As a byproduct of the construction, the testing correlation $\T M\times \TS M \tto \TT \RR$ emerged in \eqref{eq-tantest}, embedding  $\T M \hookrightarrow \RR^{\TS M}$ and $\TS M \hookrightarrow \RR^{\T M}$. Uncovering a duality of $\T M$ and $\TS M$ is in itself, of course, hardly surprising. Uncovering it through a testing correlation, akin to the 'Gedankenexperiments' whereby Turing machines give birth to complexity classes, is perhaps more interesting. Where does it lead?

In Sec.~\ref{subsec:testdyn}, we outlined a possible second application of the testing factorization \eqref{eq:epimono}, towards a coalgebraic framework for describing the families of smooth behaviors indistinguishable under a family of smooth tests. This leads towards a coalgebraic extension of Mackey's \emph{dual systems} \cite{MackeyG:dual}, studied in detail in \cite[Ch.~2-3]{GrothendieckA:TVS}, and axiomatized in \emph{Chu categories} \cite{BarrM:autonomous-book,BarrM:chu-history}. While the testing correlations considered in this paper were limited to the tangents and the cotangents on the same manifold $M$, the correlations considered in the practice are often between the tangents on one manifold, and the cotangents on another one. An interesting framework to further research in this direction seems to be the category $\Cat{STC}$, where
\begin{itemize}
\item the objects are the \emph{smooth testing correlations} 
\bear
\alpha & = &  \left(\alpha_\ast \times \alpha_\# \tto \TT \RR\right)
\eear
where
\begin{itemize}
\item $\T M \tto{\alpha_\ast} M$ is an algebra,
\item $N \tto{\alpha_\#} \TCS N$ is a coalgebra, and
\item the testing correlations $\TT$ make the following square commute
\[
\begin{tikzcd}[ampersand replacement=\&]
		\T M \times N 
		\arrow{rr}{{\rm id}\times \alpha_\#} 
		\arrow[swap]{d}{\alpha_{\ast}\times {\rm id}} 
		\&\& \T M \times \TCS N \arrow{d}{\TT}\\
M \times N 
\arrow[swap]{rr}{\TT}	
\&\&	 \RR
\end{tikzcd}
\]
\end{itemize}

\item a morphism $f: \alpha \to \beta$ is a pair $f = \left<f_\ast, f_\#\right>$ where 
\begin{itemize}
\item $\alpha_\ast \tto{f_\ast} \beta_\ast$ is an algebra homomorphism,
\item $\beta_\# \tto{f_\#} \alpha_\#$ is a coalgebra homomorphism, and
\item following square commutes
\[
\begin{tikzcd}[ampersand replacement=\&]
		M \times N' 
		\arrow{rr}{{\rm id}\times f_\#} 
		\arrow[swap]{d}{f_{\ast}\times {\rm id}} 
		\&\& M \times N \arrow{d}{\TT_\alpha}\\
M' \times N' 
\arrow[swap]{rr}{\TT_\beta}	
\&\&	 \RR
\end{tikzcd}
\]
where $M'$ is the underlying object of $\beta_\ast$ and $N'$ is the underlying object of $\beta_\#$.
\end{itemize}
\end{itemize}
The morphisms between certain integral operators as testing correlations express the Stokes-type of theorems (including the Fundamental Theorem of Calculus). The behaviors observable through such testing correlations can now be extracted by factoring out the separable and extensional objects of the category $\Cat{STC}$ \cite[Sec.~2]{BarrM:revisited}. The various algebraic invariants arise from manifolds through testing vector fields, just like the classes of languages arise through testing state machines. The question is whether this conceptual link is strong enough to support any technical hooks. 

\paragraph{Acknowledgements and apology.} We are grateful to the editors and to the anonymous reviewers for their patience and help.  The work presented here started from the connection of Mackey's dual systems and Chu categories, suggested to in a phone conversation with Bill Lawvere some 20 years ago, following the first author's presentation of the results of \cite{PavlovicD:chuI}. The approach subsequently evolved through concrete, domain specific applications in various research projects (on hybrid systems, control theory, embedded systems). The testing framework emerged from these applications. At best, the present paper brings us in the position to begin to study the original idea, sketched in the final section. While an elaboration in cartesian closed categories, along the lines of \cite{LawvereFW:smooth}, would surely have its advantages, we felt that it is important to tell the story in a language as standard as possible, as it may be of interest for several different communities  (coalgebraists, physicists, functional analysts \ldots). The balance of a presentation is, of course, always a matter of skill and taste, and we can only hope that our clumsiness will not completely obscure the clarity of the ideas that we attempted to present.

%=======================================================================
% References
%-----------------------------------------------------------------------
\addcontentsline{toc}{section}{References}
\bibliographystyle{plain}
\bibliography{TangentTesting,PavlovicD}

\providecommand{\noopsort}[1]{}
\begin{thebibliography}{10}

\bibitem{Abraham:Marsden:Ratiu:1988a}
R.~Abraham, J.E. Marsden, and T.~Ratiu.
\newblock {\em Manifolds, Tensor Analysis, and Applications}.
\newblock Applied Mathematical Sciences, Vol. 75. Springer-Verlag, New York,
  1988.

\bibitem{AczelP:NWF}
P.~Aczel.
\newblock {\em Non-well-founded Sets}.
\newblock Number~14 in Lecture Notes. Center for the Study of Language and
  Information, Stanford University, 1988.

\bibitem{SGA4}
M.~Artin, A.~Grothendieck, and J.L. Verdier.
\newblock {\em Th{\'e}orie des topos et cohomologie {\'e}tal{\'e} des
  sch{\'e}mas}, volume 269, 270, 305 of {\em Lecture Notes in Mathematics}.
\newblock Springer-Verlag, 1972.

\bibitem{BarrM:autonomous-book}
Michael Barr.
\newblock {\em $\ast$-Autonomous Categories}, volume 752 of {\em Lecture Notes
  in Mathematics}.
\newblock Springer-Verlag, 1979.

\bibitem{BarrM:revisited}
Michael Barr.
\newblock $\ast$-autonomous categories, revisited.
\newblock {\em Journal of Pure and Applied Algebra}, 111(1):1--20, 1996.

\bibitem{BarrM:chu-history}
Michael Barr.
\newblock The chu construction: history of an idea.
\newblock {\em Theory Appl. Categ.}, 17(1):10--16, 2006.

\bibitem{DeNicola:2012}
Marco Bernardo, Rocco {De Nicola}, and Michele Loreti.
\newblock Revisiting trace and testing equivalences for nondeterministic and
  probabilistic processes.
\newblock In Lars Birkedal, editor, {\em Proceedings of FOSSACS 2012}, volume
  7213 of {\em Lecture Notes in Computer Science}, pages 195--209. Springer,
  2012.

\bibitem{Blute:Cockett:Seely:differential}
Richard Blute, J.~Robin~B. Cockett, and R.~A.~G. Seely.
\newblock Differential categories.
\newblock {\em Mathematical Structures in Computer Science}, 16(6):1049--1083,
  2006.

\bibitem{Blute:Cockett:Seely:cart-differential}
Richard Blute, J.~Robin~B. Cockett, and R.~A.~G. Seely.
\newblock Cartesian differential categories.
\newblock {\em Theory and Applications of Categories}, 22(23):622--672, 2009.

\bibitem{Caenpeel:Militaru:Zhu:2002a}
S.~Caenepeel, G.~Militaru, and Shenglin Zhu.
\newblock {\em Frobenius and separable functors for generalized module
  categories and nonlinear equations}.
\newblock Springer Verlag, 2002.

\bibitem{Cockett:Cruttwell:2012a}
J.R.B. Cockett and G.S.H. Cruttwell.
\newblock Differential structure, tangent structure, and {SDG}.
\newblock {\em submitted to Applied Categorical Structures}, pages 1--73, May
  2012.

\bibitem{Cockett:Seely:2011a}
J.R.B. Cockett and R.A.G. Seely.
\newblock The {F}a\`a di {B}runo construction.
\newblock {\em Theor. and Appl. of Categories}, 25:394--425, 2011.

\bibitem{Nicola:Hennessy:84}
Rocco {De Nicola} and Matthew Hennessy.
\newblock Testing equivalences for processes.
\newblock {\em Theor. Comput. Sci.}, 34:83--133, 1984.

\bibitem{DoldA:Lectures}
A.~Dold.
\newblock {\em Lectures on Algebraic Topology}.
\newblock Classics in Mathematics. Springer Berlin Heidelberg, 1995.

\bibitem{kriegl:frolicher:1988}
A.~Fr{\"o}licher and A.~Kriegl.
\newblock {\em Linear spaces and differentiation theory}.
\newblock Pure and applied mathematics. Wiley, 1988.

\bibitem{GrothendieckA:TVS}
A.~Grothendieck.
\newblock {\em Topological Vector Spaces}.
\newblock Notes on mathematics and its applications. Gordon and Breach, 1973.

\bibitem{Haghiverdi-Pappas}
Esfandiar Haghverdi, Paulo Tabuada, and George~J. Pappas.
\newblock Bisimulation relations for dynamical, control, and hybrid systems.
\newblock {\em Theor. Comput. Sci.}, 342(2-3):229--261, 2005.

\bibitem{Harel-Kozen}
David Harel, Jerzy Tiuryn, and Dexter Kozen.
\newblock {\em Dynamic Logic}.
\newblock MIT Press, Cambridge, MA, USA, 2000.

\bibitem{Hasuo-Jacobs-Sokolova}
Ichiro Hasuo, Bart Jacobs, and Ana Sokolova.
\newblock Generic trace semantics via coinduction.
\newblock {\em Logical Methods in Computer Science}, 3(4), 2007.

\bibitem{FEA}
Allyn Jackson.
\newblock {Comme Appel\'{e} du N\'{e}ant - As If Summoned from the Void: The
  Life of Alexandre Grothendieck}.
\newblock {\em Notices of the AMS}, 51(4,10):1038--1056, 1196--1212, November
  2004.

\bibitem{Jacobs-Sokolova-exemplaric}
Bart Jacobs and Ana Sokolova.
\newblock Exemplaric expressivity of modal logics.
\newblock {\em J. Log. Comput.}, 20(5):1041--1068, 2010.

\bibitem{jubin:2012a}
Beno\^\i t~Michel Jubin.
\newblock {\em The tangent functor monad and foliations}.
\newblock PhD thesis, University of Berkeley, Berkeley, 2012.
\newblock arxiv:1401.0940.

\bibitem{Kapulkin}
Krzysztof Kapulkin, Alexander Kurz, and Jiri Velebil.
\newblock Expressiveness of positive coalgebraic logic.
\newblock In Thomas Bolander, Torben Bra{\"u}ner, Silvio Ghilardi, and
  Lawrence~S. Moss, editors, {\em Advances in Modal Logic}, pages 368--385.
  College Publications, 2012.

\bibitem{KlinB:07}
Bartek Klin.
\newblock Bialgebraic operational semantics and modal logic.
\newblock In {\em LICS}, pages 336--345. IEEE Computer Society, 2007.

\bibitem{KlinB:beyond}
Bartek Klin.
\newblock Coalgebraic modal logic beyond sets.
\newblock {\em Electr. Notes Theor. Comput. Sci.}, 173:177--201, 2007.

\bibitem{KlinB:09}
Bartek Klin.
\newblock Bialgebraic methods and modal logic in structural operational
  semantics.
\newblock {\em Inf. Comput.}, 207(2):237--257, 2009.

\bibitem{KockA:SDG}
A.~Kock.
\newblock {\em Synthetic Differential Geometry}, volume 333 of {\em London
  Mathematical Society Lecture Note Series}.
\newblock Cambridge University Press, 2006.

\bibitem{Kolar:Michor:Slovak:1993a}
I.~Kolar, P.W. Michor, and J.~Slovak.
\newblock {\em Natural Operations in Differential Geometry}.
\newblock Springer-Verlag, 1993.

\bibitem{Kupke:Kurz:Pattinson:04}
Clemens Kupke, Alexander Kurz, and Dirk Pattinson.
\newblock Algebraic semantics for coalgebraic logics.
\newblock {\em Electr. Notes Theor. Comput. Sci.}, 106:219--241, 2004.

\bibitem{Kupke:Kurz:Pattinson:05}
Clemens Kupke, Alexander Kurz, and Dirk Pattinson.
\newblock Ultrafilter extensions for coalgebras.
\newblock In Jos{\'e} Luiz~Fiadeiro et~al, editor, {\em CALCO}, volume 3629 of
  {\em Lecture Notes in Computer Science}, pages 263--277. Springer, 2005.

\bibitem{Kurz:Rosicky:12}
Alexander Kurz and Jir\'{\i} Rosick{\'y}.
\newblock Strongly complete logics for coalgebras.
\newblock {\em Logical Methods in Computer Science}, 8(3), 2012.

\bibitem{LawvereFW:smooth}
F.~William Lawvere.
\newblock Toward the description in a smooth topos of the dynamically possible
  motions and deformations of a continuous body.
\newblock {\em Cahiers de Topologie et GŽomŽtrie DiffŽrentielle CatŽgoriques},
  21(4):377--392, 1980.

\bibitem{LawvereFW:variable}
William~F. Lawvere.
\newblock {Continuously variable sets: algebraic geometry = geometric logic}.
\newblock In {\em Logic Colloquium '73 (Bristol, 1973)}, volume~80 of {\em
  Studies in Logic and the Foundations of Mathematics}, pages 135--156.
  North-Holland, Amsterdam, 1975.

\bibitem{LawvereFW:gros}
William~F. Lawvere.
\newblock Categories of spaces may not be generalized spaces as exemplified by
  directed graphs.
\newblock {\em Revista Colombiana de Matematicas}, 20:179--186, 1986.

\bibitem{Leray}
J.~Leray.
\newblock Sur la forme des espaces topologiques et sur les pointes fixes des
  repr\'esentations.
\newblock {\em J. Math. Pures et Appl.}, 9:95--249, 1945.

\bibitem{MacLane:Moerdijk:1992a}
S.~Mac~Lane and I.~Moerdijk.
\newblock {\em Sheaves in Geometry and Logic: A First Introduction to Topos
  Theory}.
\newblock Universitext. Springer-Verlag, New York, 1992.

\bibitem{MackeyG:dual}
George Mackey.
\newblock On infinite dimensional vector spaces.
\newblock {\em Trans. Amer. Math. Soc.}, 51(57):155--207, 1945.

\bibitem{mesablishvili:wisbauer:2010a}
B.~Mesablishvili and R.~Wisbauer.
\newblock {G}alois functors and entwining structures.
\newblock {\em J. Algebra}, 324:464--506, 2010.

\bibitem{mesablishvili:wisbauer:2011a}
B.~Mesablishvili and R.~Wisbauer.
\newblock Bimonads and {H}opf monads on categories.
\newblock {\em J. of K-Theory}, 7(2):349--388, 2011.

\bibitem{Milner:CC}
Robin Milner.
\newblock {\em Communication and concurrency}.
\newblock International Series in Computer Science. Prentice Hall, London,
  1989.

\bibitem{Modungo:Stefani:1978a}
M.~Modungo and G.~Stefani.
\newblock Some results on second tangent and cotangent spaces.
\newblock {\em Quaderni dell'Instituto di Matematica dell' Universit\'a di
  Lecce.}, Q.16:1--23, 1978.

\bibitem{moerdijk:2002a}
I.~Moerdijk.
\newblock Monads on tensor categories.
\newblock {\em J. Pure Appl. Algebra}, 168(2-3):189--208, 2002.

\bibitem{Moore:Gedanken}
E.F. Moore.
\newblock Gedanken experiments on sequential machines.
\newblock In C.~E. Shannon and J.~McCarthy, editors, {\em Automata Studies},
  Annals of Mathematics Studies, pages 129--153, Princeton, 1956. Princeton
  University Press.

\bibitem{vanGlabbeek-I}
Rob~J. \noopsort{Glabbeek}{van Glabbeek}.
\newblock The linear time-branching time spectrum (extended abstract).
\newblock In Jos C.~M. Baeten and Jan~Willem Klop, editors, {\em Proceedings of
  CONCUR '90}, volume 458 of {\em Lecture Notes in Computer Science}, pages
  278--297. Springer, 1990.

\bibitem{vanGlabbeek-II}
Rob~J. \noopsort{Glabbeek}{van Glabbeek}.
\newblock The linear time - branching time spectrum {II}.
\newblock In Eike Best, editor, {\em Proceedings of CONCUR '93}, volume 715 of
  {\em Lecture Notes in Computer Science}. Springer, 1993.

\bibitem{Pavlovic:Mislove:Worrell:2006a}
D.~Pavlovic, M.~Mislove, and J.B. Worrell.
\newblock Testing semantics: Connecting processes and process logic.
\newblock In M.~Johnson and V.~Vene, editors, {\em AMAST 2006}, LNCS 4019,
  pages 308--322. Springer-Verlag Berlin, 2006.
\newblock (The version with the Appendix is available from {\tt dusko.org}).

\bibitem{PavlovicD:CCPS1}
Dusko Pavlovic.
\newblock Convenient categories of processes and simulations {I}: modulo strong
  bisimilarity.
\newblock In D.~Pitt et~al., editor, {\em Category Theory and Computer Science
  '95}, volume 953 of {\em Lecture Notes in Computer Science}, pages 3--24.
  Springer Verlag, 1995.

\bibitem{PavlovicD:chuI}
Dusko Pavlovic.
\newblock Chu i: cofree equivalences, dualities and $\ast$-autonomous
  categories.
\newblock {\em Math. Structures in Comp. Sci.}, 7(2):49--73, 1997.

\bibitem{PavlovicD:IC12}
Dusko Pavlovic.
\newblock Monoidal computer {I}: {Basic computability by string diagrams}.
\newblock {\em Information and Computation}, 226:94--116, 2013.
\newblock arxiv:1208.5205.

\bibitem{PavlovicD:LAPL}
Dusko Pavlovi\'c and Mart\'{\i}n Escard\'o.
\newblock Calculus in coinductive form.
\newblock In V.~Pratt, editor, {\em Proceedings. Thirteenth Annual IEEE
  Symposium on Logic in Computer Science}, pages 408--417. IEEE Computer
  Society, 1998.

\bibitem{PavlovicD:AMAST08}
Dusko Pavlovic, Peter Pepper, and Douglas~R. Smith.
\newblock Evolving specification engineering.
\newblock In Jose Meseguer and Grigore Rosu, editors, {\em Proceedings of AMAST
  2008}, volume 5140 of {\em Lecture Notes in Computer Science}, pages
  299--314. Springer Verlag, 2008.

\bibitem{PavlovicD:MPC10}
Dusko Pavlovic, Peter Pepper, and Douglas~R. Smith.
\newblock Formal derivation of concurrent garbage collectors.
\newblock In Jules Desharnais, editor, {\em Proceedings of MPC 2010}, volume
  6120 of {\em Lecture Notes in Computer Science}, pages 353--376. Springer
  Verlag, 2010.
\newblock full version arxiv.org:1006.4342.

\bibitem{PavlovicD:ASE01}
Dusko Pavlovic and Douglas~R. Smith.
\newblock Composition and refinement of behavioral specifications.
\newblock In {\em Automated Software Engineering 2001. The Sixteenth
  International Conference on Automated Software Engineering}. IEEE, 2001.

\bibitem{PavlovicD:AMAST02}
Dusko Pavlovic and Douglas~R. Smith.
\newblock Guarded transitions in evolving specifications.
\newblock In H.~Kirchner and C.~Ringeissen, editors, {\em Proceedings of AMAST
  2002}, volume 2422 of {\em Lecture Notes in Computer Science}, pages
  411--425. Springer Verlag, 2002.

\bibitem{PavlovicD:SSAS}
Du\v{s}ko Pavlovi\'c.
\newblock Towards semantics of self-adaptive software.
\newblock In Paul~Robertson et~al, editor, {\em Proceedings of the Workshop on
  Self-Adaptive Software}, volume 1936 of {\em Lecture Notes in Computer
  Science}, pages 50--64. Springer Verlag, 2001.

\bibitem{PrattV:floyd}
Vaughan~R Pratt.
\newblock Semantical consideration on {Floyd-Hoare} logic.
\newblock In {\em 17th Annual Symposium on Foundations of Computer Science},
  pages 109--121. IEEE, 1976.

\bibitem{Rosicky:1984a}
J~Rosick\'y.
\newblock Abstract tangent functors.
\newblock {\em Diagrammes}, 12(3):JR1--JR11, 1984.

\bibitem{RuttenJ:Universal}
J.~Rutten.
\newblock Universal coalgebra: a theory of systems.
\newblock {\em Theor. Comp. Sci.}, 249:3--80, 2000.

\bibitem{Schwarz:1950a}
L.~Schwartz.
\newblock {\em Th\'eorie des distributions}.
\newblock Hermann, Paris, 1950.

\bibitem{TennisonB:Sheaf}
B.R. Tennison.
\newblock {\em Sheaf Theory}.
\newblock Cambridge Monographs on Physics. Cambridge University Press, 1975.

\bibitem{Plotkin:Turi}
Daniele Turi and Gordon~D. Plotkin.
\newblock Towards a mathematical operational semantics.
\newblock In {\em LICS}, pages 280--291. IEEE Computer Society, 1997.

\end{thebibliography}
%=======================================================================
\begin{appendix}
%=======================================================================
% Banach spaces 
%-----------------------------------------------------------------------
% !TEX root = manifolds.tex
%=======================================================================
\section{Norms and vector spaces}\label{app:Banach}

Manifolds and their relations to differential geometry and sheaves are
well studied. In preparing the paper, we used the books by MacLane and
Moerdijk~\cite{MacLane:Moerdijk:1992a}, and by Abraham, Marsden and
Ratiu~\cite{Abraham:Marsden:Ratiu:1988a}. The statements given without
proof are proved in these books. For further information about invariant
operations in differential geometry see~\cite{Kolar:Michor:Slovak:1993a}.
%=====================================================
%\subsection{Basic definitions}\label{subsec:basicDefs}
\emph{Smoothness} allows to approximate functions and such approximations
are characterized using norms (or if available inner products).
\begin{defn}
	A \emph{norm} of a real vector space $E$ is a mapping from $E$ into
	the reals $\Vert\cdot\Vert : E \rightarrow \RR :: e \mapsto \Vert e\Vert$
	such that:
	\begin{itemize}
		\item[\textsf{N1}] $\Vert e\Vert \ge 0$ for all $e\in E$, and
		   $\Vert e\Vert =0$ if and only if $e=0$ (positive definiteness).
		\item[\textsf{N2}] $\Vert\lambda e\Vert = \vert\lambda\vert\, \Vert e\Vert$,
		   for all $e\in E$ and $\lambda\in \RR$ (multiplicativity).
		\item[\textsf{N3}] $\Vert e_{1}+e_{2}\Vert \leq \Vert e_{1}\Vert
		   +\Vert e_{2}\Vert$ for all $e_{1},e_{2}\in E$ (triangle inequality).
	\end{itemize}
	$(E,\Vert\cdot\Vert)$ is a normed space. A norm complete space is called
	a \emph{Banach space}.
\end{defn}
A more special situation arises in presence of an inner product.
\begin{defn}
	A real \emph{inner product space} is a real vector space $E$ together
	with a mapping $\la \cdot,\cdot\ra : E \times E \rightarrow \RR ::
	(e_{1}, e_{2}) \mapsto \la e_{1},e_{2}\ra$ such that:
	\begin{itemize}
		\item[\textsf{I1}] $\la e,e_{1}+e_{2}\ra = \la e,e_{1}\ra + \la e,e_{2}\ra$,
		  (linearity, part 1).
		\item[\textsf{I2}] $\la e_{1},\lambda e_{2}\ra = \lambda \la e_{1},e_{2}\ra$,
		  (linearity, part 2).
		\item[\textsf{I3}] $\la e_{1},e_{2}\ra = \la e_{2},e_{1}\ra$,
		  (symmetry).
		\item[\textsf{I4}] $\la e,e\ra \geq 0$ and $\la e,e\ra=0$ if and only if $e=0$
		  (positive definiteness).  
	\end{itemize}
	Positive definiteness provides a non natural isomorphism
	$E \rightarrow \RR^{E}\cong E^{*}$ with the dual vector space $E^{*}$. If such
	an isomorphism exists the space is called \emph{reflexive}. A complete
	(w.r.t. a metric) reflexive inner product space is called \emph{Hilbert space}.
\end{defn}
The complex case does not pose any difficulty here, and one defines
Hermitian inner products using the complex involution analogously. An
inner product space is a normed space by setting $\Vert e\Vert^{2} =
\la e,e\ra$. 
\begin{defn}\label{def:metric}
	A \emph{metric} or \emph{distance} map on a set $E$ is a map $d_{E} : E \times E
	\rightarrow \RR$ such that
	\begin{itemize}
		\item[\textsf{d1}] $d(x,y)\ge 0$ non-negativity or separation axiom.
		\item[\textsf{d2}] $d(x,y)=0 \Leftrightarrow x=y$, coincidence axiom
		   (if not fulfilled $d_{E}$ is called pseudo metric).
		\item[\textsf{d3}] $d(x,y)=d(y,x)$ symmetry (if not fulfilled $d_{E}$
		    is called quasi metric).
		\item[\textsf{d4}] $d(x,z)\leq d(x,y)+d(y,z)$ triangle inequality.       
	\end{itemize}
\end{defn}
\textsf{d1} and \textsf{d2} state positive definiteness. An \emph{intrinsic metric}
$d_{M}$ on a manifold $M$ can be approximated by the length of paths 
connecting two points $x=\gamma(0)$, $y=\gamma(1)$ and $d_{M}(x,y) = 
\inf_{\gamma}\mathrm{length}(\gamma)$. Suppose $E$ has an addition
$+ : E\times E \rightarrow E$, e.g. a vector space structure, then $d$ is
\emph{translation invariant} if for all $x,y,a \in E$ holds $d(x,y)=d(x+a,y+a)$.
Any norm implies a metric by $d(x,y):=\Vert y-x\Vert$. A metric is
\emph{homogeneous} if for $r\in \RR$ and all $x,y\in E$ holds
$d(rx,r,y) = \vert r\vert d(x,y)$. A translation invariant homogeneous
metric defines a norm $\Vert x\Vert := d(x,0)$.
\begin{thm}[Riesz representation theorem]
  Let $E$ be an real inner product space. The map $e \mapsto \la\cdot,e\ra$
  is a linear norm preserving isomorphism of $E$ with $E^{*}$.	
\end{thm}
In other words, every linear form in $E^{*}$ can be realized as an inner
product with some fixed vector $e\in E$. We denote by $\Lin(E,F)$ the
linear morphisms from $E$ to $F$.

%===============================================
\section{Differential}\label{subsec:derivative}
 
%\begin{rem}\label{rem1}
	If $f : U\subset \RR^{n} \rightarrow \RR^{m}$ is differentiable, then choosing
	coordinates (bases), gives $f$ explicit as an $m$-tuple of
	functions in $n$ coordinates $(y_{1}=f^{1}(x_{1},\ldots,x_{n}),$
	$\ldots, y_{m}=f^{m}(x_{1},\ldots,x_{n}))$. The differential $\bfDk f(x)$
	at a point $x_{0}\in U$ is then given by the Jacobian matrix $J(f)\simeq \bfDk f$
	\begin{align*}
		\bfDk f(x_{0}) 
		  &:=
		  \left[\begin{array}{ccc}
		    \frac{\partial f^{1}(x)}{\partial x_{1}} & \ldots & \frac{\partial f^{1}(x)}{\partial x_{n}} \\
		    \vdots & & \vdots \\
		    \frac{\partial f^{m}(x)}{\partial x_{1}} & \ldots & \frac{\partial f^{m}(x)}{\partial x_{n}} 
		  \end{array}\right]_{x=x_{0}}
	\end{align*}
	For $m=1$ one has $y=f(x_{1},\ldots,x_{n})$ and the differential becomes
	\begin{align*}
		\bfDk f
		&= 
		  \left[\kern-1ex\begin{array}{ccc}
		    \frac{\partial f}{\partial x_{1}} & \ldots & \frac{\partial f}{\partial x_{n}}
		  \end{array}\kern-1ex\right],
		&&\textrm{applied to $e=(e^{1},\ldots,e^{n})^{t}$}&&&
		  \bfDk f(x)\cdot e &= \sum_{i=1}^{n} \frac{\partial f}{\partial x_{i}} e^{i}.
	\end{align*}
	As an example for an iterated differential $\bfD^{2}\kern-0.5ex f \in \Lin^{2}(\RR^{n},\RR)
	\cong \Lin(\RR^{n}\otimes \RR^{n},\RR)$, with $f: U\subset \RR^{n} \rightarrow \RR$
	one obtains
	\begin{align*}
		{\bfD}^{2}\kern-0.5ex f(x_{0})
		&= 
		\left[ \frac{\partial^{2} f(x)}{\partial x_{i}\partial x_{j}} \right]_{x=x_{0}}.
	\end{align*}
%\end{rem}
From this it is easy to conclude that the differential is a linear
operator on the space of sufficiently often differentiable functions. 

The tangent construction $\T$ is algebraically and geometrically simpler that $\bfD$. Think of $(u,e)$
as a vector $e$ in $E$ with base point $u$, then $\T f(u,e)=(f(u),\bfDk f(u)\cdot e)$
is the image vector in $F$ with base point $f(u)$. The map $\T$ is functorial
while $\bfD$ suffers from a morphism present to change the base point
successively. To see this let $f : U\subset E\rightarrow F$ and
$g : V\subset F \rightarrow G$ be composable $C^{1}$ maps, then
we get the
\begin{thm}(Composite mapping theorem)\label{thm:compositemapping}
 With $f,g$ as above, $g\circ f : U\subset E\rightarrow G$ is also $C^{1}$.
 Moreover
 \begin{align}
	 \T(g\circ f) &= \T(g) \circ \T(f) \nn
	 \bfD(g\circ f)(u)
	   &=
	 \bfD\kern-0.4ex g(f(u))\cdot \bfDk f(u)
	  &&&\textrm{(chain rule).}  
 \end{align}  
\end{thm}
The equation for $\bfD$ is the \emph{chain rule} and reads in coordinates
as follows:
\begin{align}
	\frac{\partial(g\circ f)^{j}(x)}{\partial x^{i}}
	&=
	\sum_{k=1}^{n} \frac{\partial g^{j}(f(x))}{\partial y^{k}} \,\frac{\partial f^{k}(x)}{\partial x^{i}}
   &&&& i= 1,\ldots, m,\quad y^{k}=f^{k}(x).
\end{align}
The chain rule implies the Leibniz rule if applied to $B\in \Lin(F_{1}\otimes F_{2},G)$
and map $f_{1}\times f_{2} : U\times U\rightarrow F_{1}\times F_{2}$.
\begin{thm}(Leibniz rule)
	Let $f_{i} : U\subset E\rightarrow F_{i}$ for $i=1,2$ be of class $C^{1}$,
	and $B\in \Lin(F_{1}\otimes F_{2},G)$. Then the mapping $B(f_{1},f_{2}) =
	B\circ (f_{1}\times f_{2}) : U\subset E\rightarrow G$ is of class $C^{1}$
	and
	\begin{align*}
		\bfDk \,(  B(f_{1},f_{2}))(u)\cdot e
		&=
		  B(\bfDk f_{1}(u)\cdot e,f_{2}(u))
		+ B(f_{1}(u),\bfDk f_{2}(u)\cdot e)		
	\end{align*}
\end{thm}
For one dimensional real vector spaces $B$ is just multiplication in
$\RR$, hence $\bfDk\,(f_{1}f_{2})(u) = \bfDk\,(f_{1}(u))f_{2}(u) + f_{1}(u)\bfDk f_{2}(u)$,
where $e=1$ is the unit vector tangent to the real line.

The tangent map can be applied to itself, hence one can define the
second tangent map $\T^{2}f = \T(\T f)$. Let $f : U\subset E\rightarrow F$,
then the second tangent map is given locally in terms of differentials as
\begin{align}\label{eq:TTf}
	\T^{2} f &: (U\times E) \times (E\times E) \rightarrow (F\times F)\times (F\times F) \\
	&(u, e_{1}, e_{2}, e_{3}) \mapsto
	 (f(u), \bfDk f(u)\cdot e_{1},
	  \bfDk f(u)\cdot e_{2}, {\bfD}^{2}f(u)\cdot (e_{1},e_{2}) + \bfDk f(u)\cdot e_{3})
	  \notag 	
\end{align}
To see this, recall that
\begin{align*}
	\T f(u,e) &= (f(u),\bfDk f(u)\cdot e).
\end{align*}
Applying $\T$ to $\T f$ yields
\begin{align*}
	\T(\T f)(u, & e_{1},e_{2},e_{3})
	 =
	(\T f(u,e_{1}) , \bfDk\, (\T f)(u,e_{1})\cdot(e_{2},e_{3})\,) \nn
	&= (f(u), \bfDk f(u)\cdot e_{1},
	    \bfDk\,\left[ (f(u), \bfDk f(u)\cdot e_{1}) \right]\cdot(e_{2},e_{3}) \,) \nn
	&= (f(u), \bfDk f(u)\cdot e_{1}, \bfDk f(u)\cdot e_{2} ,
	     {\bfDk\,}^{2}\kern-0.4ex f(u)\cdot(e_{1},e_{2})+ \bfDk f(u)\cdot e_{3}\,)
\end{align*}
where in the last equality the identity $\bfDk\,[1]=\pi_{1}$ has been
used and the Leibniz rule was applied.

To finish the setup, we consider differentials in a particular direction.
A \emph{differentiable path} is a class $C^{1}$ map from an open interval
$\gamma : I\subset \RR \rightarrow E$ into $E$, with differential
$\bfDk\, \gamma(t) \in \Lin(I,E)$. $\Lin(\RR,E)$ is identified with $E$
via $\bfDk\, \gamma(t)$ acting on $e=1$. The differential of the curve is
\begin{align*}
	\frac{\textrm{d}\gamma}{\textrm{dt}}(t) 
	&=
	\bfDk\, \gamma(t)\cdot 1 \,.
\end{align*}
Composing a $C^{1}$ curve $\gamma : I \rightarrow E$ with a $C^{1}$ map
$f : E\rightarrow F$ and using the chain rule yields
\begin{align*}
   \frac{\textrm{d}}{\textrm{dt}}(f(\gamma (t))) 
	&=
	\bfDk (f\circ \gamma)(t)\cdot 1
	 =
	\bfDk f (\gamma(t)) \cdot \frac{\textrm{d$\gamma$}}{\textrm{dt}}(t)
\end{align*}
Using the linear approximation $\gamma_{L}(t) = u +t\cdot e$ of $\gamma(t)$
one arrives at the 
\begin{thm}
	Let $f:U\subset E\rightarrow F$ be of $C^{1}$ differentiable
	at $u$, then the \emph{directional differential} exists at $u$ and
	reads
	\begin{align*}
		\frac{\textrm{d}}{\textrm{dt}}(f(u+t\cdot e))\vert_{t=0}
		&=
		\bfDk f(u)\cdot e = \bfDk_{e} f(u).
	\end{align*}
\end{thm}
If $f : \RR^{n}\rightarrow \RR$ with coordinate (basis) functions $x^{i}$ and
let $e=x^{1}e_{1}+\ldots + x^{n}e_{n}$ (usually normalized) expressed
in the standard basis $\{e_{i}\}$, then the directional differential (in
one or many directions $v_{i}$) reads
\begin{align*}
	\bfDk f(u)\cdot e = \bfDk_{e} f(u)
	&=
	\frac{\partial f}{\partial x^{1}}x^{1} + \ldots +
	\frac{\partial f}{\partial x^{n}}x^{n}
	\\
	\bfDk^{\mkern1mu k}_{v_{1},\ldots,v_{k}} f(u)
	&=
	\frac{\partial^{k}}{\partial t_{1}\cdots \partial t_{k}}
	f(u+t_{1} v_{1}+\ldots+ t_{k} v_{k})\vert_{t_{i}=0}
\end{align*}
As we compare linearizations, we can replace the open interval $I$
by $\RR$ if the reparameterization is done in such a way that base points
and tangents are preserved.

\end{appendix}
%% this file goes away after cleaning up
%\input TTclassical.tex
%=======================================================================
\end{document}